\documentclass[preprint,3p,11pt,authoryear]{elsarticle}
\usepackage{setspace}
\usepackage{booktabs}
\usepackage{multirow}
\usepackage{graphicx}
\usepackage[short,c2,nocomma]{optidef}
\usepackage{xcolor}
\usepackage{amsmath, amsthm, amsfonts}
\usepackage{algorithm}
\usepackage{algorithmicx}%
\usepackage{algpseudocode}%

\usepackage[colorlinks=true,linkcolor=blue, citecolor=blue, urlcolor=blue, breaklinks]{hyperref}

\newtheorem{proposition}{Proposition}

\newtheorem{definition}{Definition}
\newtheorem{corollary}{Corollary}

\newtheorem{theorem}{Theorem}

\begin{document}

\begin{frontmatter}



\title{Input convex neural networks as surrogates in mathematical optimisation} 

\author[1]{Yu Liu} 
\author[2]{Jan Kronqvist}
\author[3,1]{Fabricio Oliveira\corref{cor}}
\cortext[cor]{Corresponding author.}
\ead{fabol@dtu.dk}

\affiliation[1]{organization={Department of Mathematics and Systems Analysis, School of Science, Aalto University}, city={Espoo}, country={Finland}}
\affiliation[2]{organization={Department of Mathematics, KTH Royal Institute of Technology}, city={Stockholm}, country={Sweden}}
\affiliation[3]{organization={DTU Management, Technical University of Denmark}, city={Lyngby}, country={Denmark}}

\begin{abstract}
Embedding trained neural networks as surrogates within optimisation problems is an established practice in operations research. The prevailing approach uses feedforward neural networks (FNNs) with ReLU activations, whose piecewise-linear structure admits an exact but computationally intensive mixed-integer programming (MIP) reformulation as the networks grow.
We advocate input convex neural networks (ICNNs) as structurally superior surrogates when the underlying response is approximately convex or concave. The convex architecture offers two computational advantages. First, the ICNN-MIP formulation tends to yield a tighter linear programming (LP) relaxation than its FNN-MIP counterpart, with no integrality gap in favourable instances. Second, ICNNs uniquely admit an LP-based reformulation via epigraph representations of ReLU activations, though this embedding is not always exact. 
When it is not, we exploit the properties of ICNNs to construct the strongest continuous relaxation over box domains, namely, the convex hull of the ICNN's graph, bounded below by the epigraph and above by the concave envelope; this construction is tractable under input convexity but hard for general ReLU networks. 
On this basis, we develop a branch-and-bound algorithm that builds this relaxation at each node, branches directly on input variables rather than intermediate variables as in MIP reformulations, and terminates at the root node whenever the epigraph embedding is valid. 
Case studies on humanitarian food aid, oil well routing, and wine blending show that ICNN surrogates match FNN accuracy and deliver gains in solve time and scalability, supporting ICNN as the default surrogate when the underlying function is convex, concave, or well-approximated as such.
\end{abstract}


\begin{highlights}
    \item ICNN-MIP admits exact LP relaxations when minimising the ICNN output.
    \item Epigraph and concave envelope bound the ICNN, giving its strongest relaxation.
    \item Together, they yield the ICNN graph's convex hull, unavailable for general NNs.
    \item A branch-and-bound exploits this relaxation and branches directly on input variables.
    \item Case studies confirm faster, more scalable solutions using ICNN surrogates.
\end{highlights}


\end{frontmatter}

\section{Introduction}

Neural networks (NNs) have become an integral part of optimisation problems, particularly as surrogate models. NNs serve as reduced-order substitutes for complex functions that are accessible only via sampling and possess completely (``black box'') or partially (``grey box'') unknown structures~\citep{tsay_sobolev_2021,liu_simulator-based_2025}. Due to their accuracy and computational efficiency, NN surrogates are applied across a wide range of domains in operations research, including process engineering~\citep{addis_data_2023}, vehicle routing~\citep{wang_surrogate-accelerated_2025}, and  scheduling~\citep{fons_moreno-palancas_solving_2025}. 

The role of the surrogate in the optimisation model depends critically on how its inputs are treated.  This paper focuses on the more demanding setting where the surrogate inputs \emph{include} the decision variables, as arises when first-principles models are unavailable or too costly to evaluate. In this setting, the surrogate must be \emph{embedded} within the optimisation problem, either implicitly or explicitly.

Implicit embedding methods pass gradients or Hessians directly to nonlinear programming (NLP) solvers via back-propagation~\citep{chen_data-driven_2020, dowson_mathoptaijl_2026} and have been used extensively in nonlinear model predictive control~\citep{gokhale_physics_2022, lawrynczuk_input_2022, zheng_physics-informed_2023}. However, implicit methods provide no global optimality guarantees and can suffer from numerical instability. In contrast, explicit embedding encodes the surrogate model directly as constraints in the optimisation model, enabling the use of global solvers.

The dominant explicit embedding approach relies on feedforward neural networks (FNNs) with Rectified Linear Unit (ReLU) activations~\citep{huchette_when_2026}. Their piecewise-linear structure admits an exact mixed-integer programming (MIP) reformulation, in which each ReLU unit is encoded with a binary variable and big-$M$ constraints~\citep{cheng_maximum_2017, fischetti_deep_2018, serra_bounding_2018, tjeng_evaluating_2019}. The tightness of the linear programming (LP) relaxation of the resulting MIP is critical for solver performance, motivating stronger formulations based on convex hulls~\citep{anderson_strong_2020} or $P$-split formulations~\citep{kronqvist_p-split_2025}. 

Despite these advances, the FNN-MIP approach has a fundamental scalability limitation: every ReLU unit requires one binary variable, so the number of binary variables grows linearly with the total number of neurons. Compounding this, the continuous relaxation of the resulting MIP is inherently weak~\citep{tong_optimization_2024}, making the 0-1 MIP computationally prohibitive as network size increases. Acceleration techniques such as bound tightening~\citep{grimstad_relu_2019,wang_optimizing_2023}, model regularisation~\citep{kenefake_multiparametric_2024}, and model compression~\citep{cheng_survey_2020} can reduce computational burden, but the NP-hard nature of MIP remains a fundamental barrier for larger networks.

Input convex neural networks (ICNNs), proposed by~\citet{amos_input_2017}, offer an alternative approach when the underlying phenomenon being modelled is approximately convex. By imposing specific architectural constraints, ICNNs encode a convex relationship between inputs and outputs. This structural property enables exact inference via an LP-based reformulation that uses an epigraph representation of ReLU activations, without introducing binary variables, which is particularly valuable when the learned model must be embedded within a larger optimisation problem. 

Prior work on ICNNs falls broadly into two streams. The first concerns standalone function approximation and system identification, where ICNNs serve purely as predictive models, with examples including learning Optimal Power Flow (OPF) value functions~\citep{chen_data-driven_2020, rosemberg_learning_2024}, optimal transport between two distributions~\citep{makkuva_optimal_2020}, quantifying shipboard fuel efficiency~\citep{kim_ship_2026}, and learning system dynamics for use in model predictive control~\citep{yang_optimization-based_2021, bunning_input_2021, jiang_path-following_2022, wang_fast_2024}. The second, more recent stream concerns ICNN-embedded optimisation, distinguished by its explicit use of LP representability to incorporate trained networks directly as constraints within optimisation problems~\citep{dvorkin_emission-constrained_2023, wu_transient_2024, zhu_day-ahead_2025, liu_icnn-enhanced_2025}.

In this paper, we investigate the use of ICNNs as surrogate models embedded within mathematical optimisation problems, where the surrogate inputs include the decision variables. While the LP representability of ICNNs is well established, two fundamental questions remain unaddressed in the literature. 
The first concerns the validity of the epigraph embedding, including the conditions under which it yields exact reformulation, as well as how to obtain convergence to the true optimum when validity fails. The second concerns the comparative tightness of ICNN-MIP versus FNN-MIP LP relaxations, which has not been rigorously established despite MIP reformulation remaining the predominant practical approach. We address these gaps and make the following contributions:%
\begin{enumerate}
    \item Structural analysis of ICNN-MIP formulations: We characterise exactness properties of the ICNN-MIP LP relaxation (Proposition~\ref{prop:exactness}) and compare the resulting relaxation gap with that of FNN-MIP formulations.
    \item Formal definition of epigraph embedding validity: We give a precise definition of when the epigraph embedding of an ICNN is a valid reformulation of the original problem (Definition~\ref{def:validity}) and discuss the problem structures under which validity holds or fails.
    \item The strongest continuous relaxation of ICNN surrogates: We construct a concave envelope of the ICNN from vertex evaluations (Proposition~\ref{prop:concave_env}) and prove that, combined with the epigraph, it yields the convex hull of the ICNN's graph over box domains, the strongest continuous relaxation of the ICNN formulation (Theorem~\ref{thm:strongest_relax}). No counterpart is known for general ReLU network formulations.
    \item A specialised branch-and-bound (BB) algorithm: Building on this relaxation, we develop ICNN-BB~(Algorithm~\ref{alg:BB}), an LP-based BB algorithm for ICNN-embedded optimisation that avoids introducing binary variables. The algorithm branches directly on the input variables rather than on intermediate variables encoding the network, and terminates at the root node whenever the epigraph embedding is valid.
    \item Empirical validation on three case studies: We evaluate the proposed approaches across three problems from the literature, confirming tighter MIP relaxations for ICNN-MIP, root-node termination of ICNN-BB when the epigraph embedding is valid, and superior performance of ICNN-BB where the vertex-based envelope remains manageable.
\end{enumerate}

The structure of the paper is as follows. Section~\ref{sec:pre} reviews FNN and ICNN fundamentals and introduces the general setting of embedding trained NNs in optimisation problems. Section~\ref{sec:icnn_mip} provides a structural comparison of ICNN-MIP and FNN-MIP LP relaxations. Section~\ref{sec:icnn_bb} formalises the validity of the epigraph embedding, constructs the strongest continuous relaxation of the ICNN formulation, and develops the specialised BB algorithm built on it. Section~\ref{sec:exp} describes the three case studies and analyses computational results. Section~\ref{sec:con} concludes the findings and limitations.
    
\section{Preliminaries}
\label{sec:pre}

First, we briefly introduce the architectures of FNNs and ICNNs and their exact reformulations as mathematical optimisation problems. The training of ICNNs is also discussed. Finally, we present the general setting in which trained NNs are embedded in optimisation problems, and the two computational routes for doing so.

\subsection{Feedforward neural networks}\label{sec:pre_fnn}

FNNs consist of consecutive layers made up of individual units called neurons, each of which applies an activation function to a weighted sum of the previous layer's outputs. Although we consider only networks with a single output neuron, the techniques can be easily extended to handle multiple outputs.

\subsubsection{Mixed-integer formulation}

An FNN with piecewise-linear activations can be exactly reformulated as a mixed-integer program (MIP). We consider networks with ReLU activations in the hidden layers and an identity function in the output layer. Following~\citet{fischetti_deep_2018}, we formulate the FNN-MIP as~\eqref{eq:nn_mip}. Here $W_k$ and $b_k$ denote the weight matrix and bias vector of layer $k$, and $x_k$ is the vector of post-activation values at layer $k$. For each neuron $j$ in the hidden layer $k$, a binary variable $\delta_{k,j} \in \{0,1\}$ indicates activation state and a slack variable $s_{k,j} \geq 0$ captures the negative part of the pre-activation value. The big-$M$ constants $U_{k,j}$ and $L_{k,j}$ in constraints~\eqref{milp:con2} and~\eqref{milp:con3} are valid upper and lower bounds on the pre-activation value of neuron $j$ in layer $k$. Constraint~\eqref{milp:con4} bounds the input variables to a specified domain.%
\begin{mini!}
    {}{0 \label{eq:obj}}{\label{eq:nn_mip}}{}
    \addConstraint{W_{k} x_{k} + b_{k}}{= x_{k+1} - s_{k+1} \quad}{k=0,\ldots,K-2 \label{milp:con1}}
    \addConstraint{W_{K-1} x_{K-1} + b_{K-1}}{= x_{K} \quad}{\label{milp:con7}}
    \addConstraint{x_{k,j}}{\leq U_{k,j} \cdot \delta_{k,j} \quad}{k=1,\ldots,K-1, \ j=1,\ldots,n_k \label{milp:con2}}
    \addConstraint{s_{k,j}}{\leq -L_{k,j} \cdot (1-\delta_{k,j}) \quad}{k=1,\ldots,K-1 ,\ j=1,\ldots,n_k \label{milp:con3}}
    \addConstraint{L_0 \leq x_0}{\leq U_0 \label{milp:con4}}
    \addConstraint{\delta_{k,j}}{\in \{0, 1\} \quad}{k=1,\ldots,K-1 ,\ j=1,\ldots,n_k \label{milp:con5}}
    \addConstraint{x_{k,j},\, s_{k,j}}{\geq 0 \quad}{k=1,\ldots,K-1 ,\ j=1,\ldots,n_k. \label{milp:con6}}
\end{mini!}

The zero objective reflects that the formulation encodes the network structure as a system of constraints. For a fixed input $x_0$, the constraints uniquely determine the continuous forward-pass variables, recovering the network output $x_K$. When using FNNs as surrogates, appending these constraints to a larger optimisation problem is sufficient to embed the network.

\subsection{Input convex neural networks}\label{sec:pre_icnn}

\subsubsection{ICNN architectures}

ICNNs have a similar structure to FNNs with the exception that, during training, the layer weights are constrained to be non-negative (except for the first layer), and the activation function $\sigma(\cdot)$ must be convex and non-decreasing, a condition satisfied by ReLU. These modifications make the output of ICNNs convex with respect to the inputs, as proved by~\citet{amos_input_2017}. In addition, weighted skip connections from the input layer to each of the subsequent layers are implemented to improve representation capacity while preserving convexity. The simplified structure is illustrated in Figure~\ref{fig:icnn_struct}. Here, $z_0$ denotes input vectors, $W_k$ the ``regular'' layer weights and $S_k$ the skip connection weights.

\begin{figure}[ht!]
    \centering
    \includegraphics[width=0.45\linewidth]{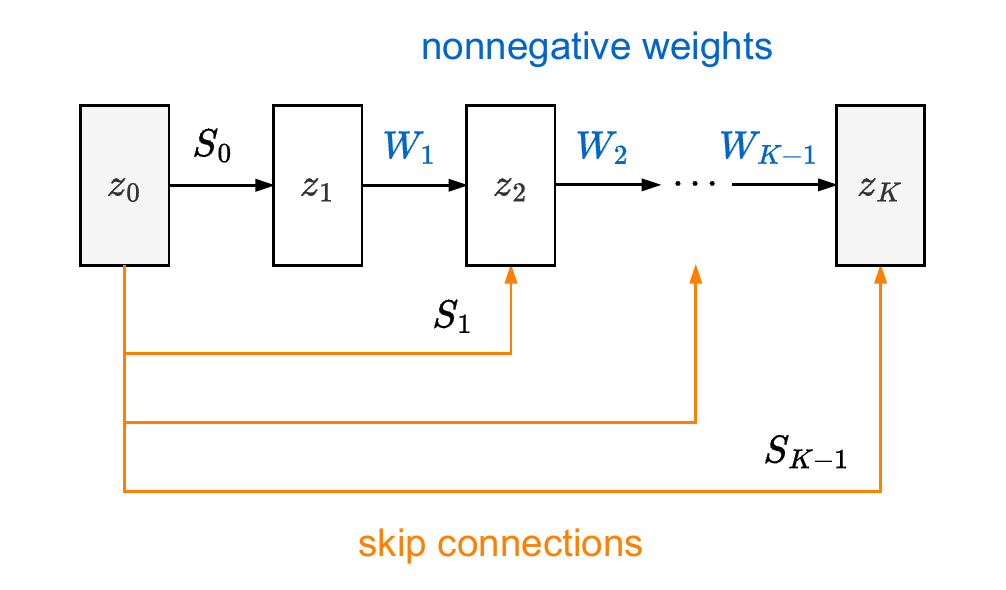}
    \caption{A simplified illustration of the ICNN architecture, reproduced from~\citet{liu_icnn-enhanced_2025}.}
    \label{fig:icnn_struct}
\end{figure}

The mapping for each layer can be expressed generically as:%
\begin{equation}
    z_{k+1} = \sigma ( W_k z_k + S_k z_0 + b_k), \quad k = {0, \dots, K-1},
\label{eq:fcicnn}
\end{equation}
where $z_k$ for $k \geq 1$ is obtained by applying an activation function $\sigma(\cdot)$, with $W_0 \equiv 0$ by convention. Since $W_{1:K-1} \geq 0$ and $\sigma(\cdot)$ is convex and non-decreasing, the convexity of $\hat{f}$ in $z_0$ is preserved layer by layer. Indeed, a non-negative weighted sum of convex functions remains convex, and composing it with a convex non-decreasing function retains this property. Importantly, the skip connections introduce only linear terms in $z_0$, which do not violate convexity.

\subsubsection{Exact inference via LP}

The key distinction between ICNNs and FNNs is that ICNN inference admits an exact LP representation. For a fully connected ICNN with ReLU activations in the hidden layers and an identity output layer, as described in~\eqref{eq:fcicnn}, performing inference for a given input $z_0$ is equivalent to solving the following LP~\citep{amos_input_2017}:%
\begin{mini!}
    {z_1, \ldots, z_K}{z_K \label{icnn_lp:obj}}{\label{icnn_lp}}{}
    \addConstraint{z_{k+1}}{\geq W_k z_k+S_k z_0+b_k, \quad}{k=0,\ldots,K-1 \label{icnn_lp:con1}}
    \addConstraint{z_k}{\geq 0, \quad}{k=1,\ldots,K-1. \label{icnn_lp:con2}}
\end{mini!}

Formulation~\eqref{icnn_lp} adopts an \textit{epigraph} representation of the ReLU activations, replacing the layer-wise nonlinear equality constraints with linear inequalities. The minimisation objective ensures tightness by enforcing that, at optimality, at least one constraint in~\eqref{icnn_lp:con1} or~\eqref{icnn_lp:con2} must be active, which recovers the exact feedforward computation. This LP structure eliminates the need for binary variables, which is the computational basis for embedding ICNNs in larger optimisation problems without the overhead of an MIP, in contrast to the formulation required for FNNs.

\subsubsection{Practical considerations for training ICNNs}

ICNNs are trained via constrained optimisation, where the non-negativity of the weights is enforced explicitly. This can be handled efficiently and reliably using projected gradient methods, in which each gradient update is followed by a projection onto the non-negative orthant, i.e., clipping negative weights to zero. This projection step can slow convergence, as gradient updates may be partially nullified when weights are clipped. In practice, this effect can often be mitigated by moderately increasing the learning rate, though the benefit is problem-dependent and care must be taken to avoid instability. 

An important consideration before training is the nature of the target function. ICNNs can be expected to achieve good fit values only on data that suggest a convex or nearly convex relationship between inputs and outputs. Concave targets can be handled by negating the output before training. When the shape of the input-output relationship is not known in advance, a simple preliminary diagnostic with compact networks suffices. This involves training an FNN on the original target and two ICNNs, one on the original target and one on the negated target. Since FNNs are universal approximators~\citep{hornik_multilayer_1989}, the FNN's validation accuracy serves as a benchmark for what is achievable in terms of representation using ICNNs instead. Whichever ICNN comes closest to this benchmark indicates whether the target is approximately convex or concave, and a small gap suggests that the convexity restriction is not too limiting. Once this is established, the chosen ICNN can be scaled up for the final surrogate.

\subsection{Embedding trained NNs in optimisation problems}
\label{sec:pre_embedding}

We now describe the general setting in which trained NN surrogates are used within optimisation problems. Let $f:\mathbb{R}^n\to\mathbb{R}$ be a black-box mapping observed on data. A network $\hat f$ is trained to approximate $f$ and then embedded in a downstream optimisation model%
\begin{mini!}
    {x,y}{h(x,y) }{\label{eq:orig}}{}
    \addConstraint{g(x,y)}{ \leq 0, \quad}{}
    \addConstraint{y}{= \hat f(x), \quad}{}
    \addConstraint{x}{\in \mathcal{X}, \quad}{}
\end{mini!}
where $\mathcal{X} = \{x \in \mathbb{R}^n : l_j \leq x_j \leq u_j,\; j = 1,\dots,n\}$ is a box. Two cases of~\eqref{eq:orig} are of interest. In the \textit{direct output minimisation} case, $h(x,y)=y$ and the constraint $g$ is absent, so~\eqref{eq:orig} amounts to minimising the surrogate output over its input domain. In the \textit{general embedding} case, the surrogate output has a coupling with the remaining decision variables via the objective $h$ and constraint set $g$.

Two computational routes are available for representing $y = \hat f(x)$ within~\eqref{eq:orig}. The first applies to any ReLU network, FNN or ICNN, and encodes the exact forward pass with binary activation variables, as in the MIP formulation~\eqref{eq:nn_mip}. Section~\ref{sec:icnn_mip} derives the corresponding ICNN-MIP formulation and analyses its structure. The second route is specific to ICNNs and replaces each ReLU by the epigraph representation underlying~\eqref{icnn_lp}, yielding an embedding that involves only linear constraints and no binary variables. This LP route, however, comes with a caveat. The epigraph representation enforces only $y \geq \hat f(x)$, and the exactness of such reformulation stems from its objective, which penalises any overestimation of the network output. A general embedding need not penalise such overestimation. The epigraph embedding therefore constitutes a relaxation of~\eqref{eq:orig} and is only \textit{equivalent} when this relaxation is tight (or exact), a notion we formalise and analyse in Section~\ref{sec:validity}.

\section{Structural analysis of ICNN-MIP formulations}
\label{sec:icnn_mip}

While ICNNs are typically associated with their favourable LP-based inference, their structural properties offer significant, yet underexplored, advantages for MIP reformulations. This section develops the first computational route introduced in Section~\ref{sec:pre_embedding}, deriving such ICNN-MIP formulations and analysing their structural properties. 

\subsection{Derivation of ICNN-MIP formulations}

We first derive the MIP formulation for embedding ICNNs, whose exactness holds unconditionally. The architectural features that distinguish ICNNs from FNNs do not alter the piecewise-linear nature of the function represented by the network when ReLU activations are used. As a result, an ICNN can be formulated as an MIP by adopting the same big-$M$ formulation as presented for FNNs in Section~\ref{sec:pre_icnn}, with the only addition being the $S_k z_0$ term arising from the input-layer skip connections. Recall the layer-wise propagation of an ICNN defined in~\eqref{eq:fcicnn}. As in the FNN-MIP formulation~\eqref{eq:nn_mip}, binary variables $\delta_{k,j} \in \{0, 1\}$ indicate the activation state of the $j$-th neuron in the hidden layer $k$, and slack variables $s_{k,j} \geq 0$ capture the negative part of the pre-activation value. The ICNN-MIP formulation is as follows:%
\begin{mini!}
    {}{\gamma z_K \label{eq:icnn_obj}}{\label{eq:icnn_mip}}{}
    \addConstraint{W_{k} z_{k} + S_{k} z_0 + b_{k}}{= z_{k+1} - s_{k+1} \quad}{k=0,\ldots,K-2 \label{icnn_mip:con1}}
    \addConstraint{W_{K-1} z_{K-1} + S_{K-1} z_0 + b_{K-1}}{= z_{K} \quad}{ \label{icnn_mip:con7}}
    \addConstraint{z_{k,j}}{\leq U_{k,j} \cdot \delta_{k,j} \quad}{k=1,\ldots,K-1, \ j=1,\ldots,n_k \label{icnn_mip:con2}}
    \addConstraint{s_{k,j}}{\leq -L_{k,j} \cdot (1-\delta_{k,j}) \quad}{k=1,\ldots,K-1 ,\ j=1,\ldots,n_k \label{icnn_mip:con3}}
    \addConstraint{L_0 \leq z_0}{\leq U_0 \label{icnn_mip:con4}}
    \addConstraint{\delta_{k,j}}{\in \{0, 1\} \quad}{k=1,\ldots,K-1 ,\ j=1,\ldots,n_k \label{icnn_mip:con5}}
    \addConstraint{z_{k,j},\, s_{k,j}}{\geq 0 \quad}{k=1,\ldots,K-1 ,\ j=1,\ldots,n_k. \label{icnn_mip:con6}}
\end{mini!}
Here, $\gamma$ is a scalar parameter. When $\gamma=0$, \eqref{eq:icnn_mip} becomes a feasibility problem. That is, when the input $z_0$ is fixed, the formulation simply enforces the exact forward pass of the trained network for $z_0$. When $\gamma>0$, the formulation minimises the ICNN output with cost coefficient $\gamma$, corresponding, up to a positive scaling, to the direct output minimisation case of~\eqref{eq:orig}. A negative $\gamma$ would instead maximise the output, an instance of the general embedding case of~\eqref{eq:orig}. We restrict attention to $\gamma \geq 0$ in the LP relaxation analysis below, and return to the general embedding case at the end of Section~\ref{sec:icnn_fnn_mip_comparison}.

The resulting ICNN-MIP formulation~\eqref{eq:icnn_mip} is structurally equivalent to the FNN-MIP in~\eqref{eq:nn_mip}, with the addition of input-layer skip connections through the terms $S_k z_0$ and the non-negativity of the weights $W_{1:K-1}$. As in the case of FNN-MIPs, these constraints can be embedded to represent ICNNs as surrogates within larger optimisation problems.

\subsection{Exactness properties of the ICNN-MIP relaxation}

A critical performance determinant for MIPs is the strength of their LP relaxation, obtained by relaxing integrality constraints on the binary activation variables. Weaker relaxations tend to result in larger search trees and poorer computational performance, whereas tighter relaxations provide stronger bounds and can reduce the amount of branching required.

Let $\mathcal P_{\mathrm{MIP}}$ denote the feasible region of the ICNN-MIP formulation~\eqref{eq:icnn_mip}, and let $\mathcal P_{\mathrm{LP}}$ denote the feasible region obtained from~\eqref{eq:icnn_mip} by replacing $\delta_{k,j}\in\{0,1\}$ with $\delta_{k,j}\in[0,1]$. For a fixed input $z_0=\bar z_0$, the MIP determines the output value $z_K$ uniquely, independently of the value of $\gamma$, whereas the LP relaxation may admit multiple feasible values of $z_K$ and return an arbitrary feasible point if $\gamma=0$. To assess whether the relaxation would underestimate the network output, we minimise $\gamma z_K$ over $\mathcal P_{\mathrm{LP}}$ with $\gamma>0$, thereby identifying the smallest feasible value of $z_K$. Since any positive $\gamma$ affects only the objective and not the feasible region, we set $\gamma=1$ without loss of generality in the following analysis.

\begin{proposition}[Exactness of the ICNN-MIP relaxation with fixed input]\label{prop:exactness}
Let $\hat f_{\mathrm{ICNN}}:\mathbb R^{n}\to\mathbb R$ be an ICNN with ReLU activations, bounded input domain $\mathcal Z$, and non-negative weights $W_{1:K-1}$. Then, for every fixed input $\bar z_0\in\mathcal Z$, the minimum output value over the LP relaxation of~\eqref{eq:icnn_mip} equals the exact ICNN output obtained from the MIP, i.e.,%
    \begin{equation}
        \min\{z_K : (z,s,\delta)\in\mathcal P_{\mathrm{LP}},\ z_0=\bar z_0\}
        =
        \min\{z_K : (z,s,\delta)\in\mathcal P_{\mathrm{MIP}},\ z_0=\bar z_0\}
        =
        \hat f_{\mathrm{ICNN}}(\bar z_0).
        \label{eq:fixed_input_mip_lp_equivalence}
    \end{equation}
Thus, when $z_0$ is fixed, the LP relaxation of~\eqref{eq:icnn_mip} has zero integrality gap.
\end{proposition}

\begin{proof}
Fix $\bar z_0\in\mathcal Z$, and let $\bar z_k$ denote the activation vector obtained by the exact ICNN forward pass. Consider any feasible point of the LP relaxation with $z_0=\bar z_0$. The equality constraints~\eqref{icnn_mip:con1} imply
    \begin{equation}
        W_{k} z_{k} + S_{k}\bar z_0 + b_{k} = z_{k+1}-s_{k+1}, \qquad k=0,\ldots,K-2.
        \label{eq:lp_balance}
    \end{equation}
Since $s_{k+1}\geq 0$ and $z_{k+1}\geq 0$, we obtain
    \begin{equation}
        z_{k+1} \geq \max\{0,W_k z_k + S_k\bar z_0 + b_k\}, \qquad k=0,\ldots,K-2,
        \label{eq:lp_relu_lower_bound}
    \end{equation}
componentwise.

We now show by induction that every feasible point of the LP relaxation satisfies $z_k \geq \bar z_k$ componentwise. For the first hidden layer, using $W_0\equiv 0$,%
    \begin{equation}
        z_1 \geq \max\{0,S_0\bar z_0+b_0\} = \bar z_1.
        \label{eq:first_layer_dominance}
    \end{equation}
For the induction step, suppose that $z_k\geq \bar z_k$ for some $k=1, \ldots, K-2$. Since $W_{1:K-1}\geq 0$ and the ReLU is non-decreasing, it follows that
    \begin{equation}
        z_{k+1} \geq \max\{0,W_k z_k+S_k\bar z_0+b_k\} \geq \max\{0,W_k\bar z_k+S_k\bar z_0+b_k\} = \bar z_{k+1}. \label{eq:induction_step}
    \end{equation}
For the identity output layer~\eqref{icnn_mip:con7}, since $W_{K-1}\geq 0$ and $z_{K-1}\geq \bar z_{K-1}$, we have
    \begin{equation}
        z_K \geq W_{K-1}\bar z_{K-1}+S_{K-1}\bar z_0+b_{K-1} = \bar z_K = \hat f_{\mathrm{ICNN}}(\bar z_0).
        \label{eq:lp_no_underestimation}
    \end{equation}
Therefore, no LP-feasible point can attain an output below the exact ICNN value.

The exact forward-pass activations are feasible for~\eqref{eq:icnn_mip} with a valid binary activation pattern, and therefore also feasible for its LP relaxation. Hence the value $z_K=\hat f_{\mathrm{ICNN}}(\bar z_0)$ is attainable. Since we are minimising $z_K$ and have already shown that no LP-feasible point can satisfy $z_K<\hat f_{\mathrm{ICNN}}(\bar z_0)$, the LP optimum is exactly $\hat f_{\mathrm{ICNN}}(\bar z_0)$. The same value is attained by solving~\eqref{eq:icnn_mip}, so the two formulations have the same minimum output value and, thus,~\eqref{eq:fixed_input_mip_lp_equivalence} follows.
\end{proof}

Proposition~\ref{prop:exactness} shows that the LP relaxation of the ICNN-MIP recovers the exact network output at every fixed input. We now use this pointwise result to connect the LP relaxation of~\eqref{eq:icnn_mip} to the epigraph formulation~\eqref{icnn_lp}. 

\begin{corollary}[Equivalence of the LP relaxation and the epigraph formulation]
\label{cor:epigraph_connection}
For each fixed input $\bar z_0\in\mathcal Z$, define the value function induced by the LP relaxation of~\eqref{eq:icnn_mip} as
\begin{equation}
    \phi_{\mathrm{LP}}(\bar z_0) = \min\{z_K : (z,s,\delta)\in\mathcal P_{\mathrm{LP}},\ z_0=\bar z_0\}.
    \label{eq:lp_relaxation_value_function}
\end{equation}
Let $\mathcal P_{\mathrm{epi}}$ denote the feasible region of the epigraph formulation~\eqref{icnn_lp}. For every fixed input $\bar z_0\in\mathcal Z$, the LP relaxation of~\eqref{eq:icnn_mip} and the epigraph formulation~\eqref{icnn_lp} have the same minimum output value:
\begin{equation}
    \phi_{\mathrm{LP}}(\bar z_0)
    =
    \min\{z_K : z\in\mathcal P_{\mathrm{epi}},\ z_0=\bar z_0\}
    =
    \hat f_{\mathrm{ICNN}}(\bar z_0).
    \label{eq:lp_epigraph_pointwise_equivalence}
\end{equation}
Consequently, collecting these pointwise values over $\mathcal Z$ generates the epigraph of the ICNN surrogate:
\begin{equation}
    \operatorname{epi}_{\mathcal Z}(\hat f_{\mathrm{ICNN}})
    =
    \{(z_0,\eta): z_0\in\mathcal Z,\ \eta \geq \phi_{\mathrm{LP}}(z_0)\}
    =
    \{(z_0,\eta): z_0\in\mathcal Z,\ \eta \geq \hat f_{\mathrm{ICNN}}(z_0)\}.
    \label{eq:epigraph_from_lp_relaxation}
\end{equation}
\end{corollary}

\begin{proof}
By Proposition~\ref{prop:exactness}, $\phi_{\mathrm{LP}}(\bar z_0)=\hat f_{\mathrm{ICNN}}(\bar z_0)$ for every fixed $\bar z_0\in\mathcal Z$. By the exact-inference property of the epigraph formulation~\eqref{icnn_lp}, established in Section~\ref{sec:pre_icnn}, solving~\eqref{icnn_lp} with $z_0=\bar z_0$ also gives the exact ICNN output. Hence
\begin{equation}
    \phi_{\mathrm{LP}}(\bar z_0)
    =
    \min\{z_K : z\in\mathcal P_{\mathrm{epi}},\ z_0=\bar z_0\}
    =
    \hat f_{\mathrm{ICNN}}(\bar z_0),
\end{equation}
which proves~\eqref{eq:lp_epigraph_pointwise_equivalence}.

Finally, applying the pointwise identity $\phi_{\mathrm{LP}}(z_0)=\hat f_{\mathrm{ICNN}}(z_0)$ for all $z_0\in\mathcal Z$ gives
\begin{equation}
    \{(z_0,\eta): z_0\in\mathcal Z,\ \eta\geq
    \phi_{\mathrm{LP}}(z_0)\}
    =
    \{(z_0,\eta): z_0\in\mathcal Z,\ \eta\geq
    \hat f_{\mathrm{ICNN}}(z_0)\},
\end{equation}
which is the epigraph of $\hat f_{\mathrm{ICNN}}$ over $\mathcal Z$.
\end{proof}

Corollary~\ref{cor:epigraph_connection} establishes that the pointwise minimum outputs of the LP relaxation of~\eqref{eq:icnn_mip}, collected over the full input domain $\mathcal Z$, generate the epigraph of $\hat f_{\mathrm{ICNN}}$ represented by~\eqref{icnn_lp}. We next consider the corresponding optimisation problem in which $z_0$ is itself a decision variable.

\begin{corollary}[Exactness of the ICNN-MIP relaxation under direct output minimisation]
\label{cor:direct_min_exactness}
Consider the embedded optimisation problem in which the ICNN surrogate is minimised directly over its input domain. Embedding the ICNN via the LP relaxation of~\eqref{eq:icnn_mip} gives the same optimal objective value as embedding it via the exact MIP~\eqref{eq:icnn_mip},
\begin{equation}
    \min\{z_K : (z,s,\delta)\in\mathcal{P}_{\mathrm{LP}}\}
    =
    \min\{z_K : (z,s,\delta)\in\mathcal{P}_{\mathrm{MIP}}\}
    =
    \min_{z_0\in\mathcal{Z}}\hat{f}_{\mathrm{ICNN}}(z_0).
    \label{eq:direct_min_lp_mip_equivalence}
\end{equation}
\end{corollary}

\begin{proof}
By Corollary~\ref{cor:epigraph_connection}, the pointwise value induced by the LP relaxation satisfies $\phi_{\mathrm{LP}}(z_0)=\hat f_{\mathrm{ICNN}}(z_0)$ for every $z_0\in\mathcal Z$. Minimising these pointwise values over $\mathcal Z$ gives
\begin{equation}
    \min\{z_K : (z,s,\delta)\in\mathcal{P}_{\mathrm{LP}}\}
    =
    \min_{z_0\in\mathcal{Z}}\phi_{\mathrm{LP}}(z_0)
    =
    \min_{z_0\in\mathcal{Z}}\hat{f}_{\mathrm{ICNN}}(z_0).
\end{equation}
The same identity holds for $\mathcal{P}_{\mathrm{MIP}}$ because \eqref{eq:icnn_mip} encodes the ICNN forward pass exactly.
\end{proof}

Corollary~\ref{cor:direct_min_exactness} is the precise sense in which the LP relaxation is exact for direct minimisation of the ICNN output. Fractional activation patterns may exist in $\mathcal{P}_{\mathrm{LP}}$, but none can improve the objective below the true ICNN function value. Thus, the binary activation constraints can be relaxed without changing the optimal objective value.

\subsection{Relaxation-gap comparison with FNN-MIP formulations}
\label{sec:icnn_fnn_mip_comparison}

The results above establish a chain of equivalences among the ICNN forward pass, the ICNN-MIP formulation~\eqref{eq:icnn_mip}, its LP relaxation, and the epigraph formulation~\eqref{icnn_lp}, which all agree on the minimum output value, whether the input is fixed or optimised directly. Proposition~\ref{prop:exactness} shows that, when $z_0$ is fixed, minimising $z_K$ over the LP relaxation recovers the exact ICNN value. Corollary~\ref{cor:epigraph_connection} connects this property to the epigraph representation, and Corollary~\ref{cor:direct_min_exactness} shows that the relaxation remains exact when the ICNN output is minimised directly over its input domain.

This exactness property does not extend to FNN-MIP formulations. We now compare ICNN-MIP and FNN-MIP formulations from the perspective of their integrality gaps. This is the relevant comparison in the setting where ICNN and FNN surrogates are trained separately on the same input-output data and then embedded into the same downstream optimisation problem. Since the trained networks generally represent different functions, their absolute LP relaxation bounds are not directly comparable as formulation-level quantities. The meaningful comparison is instead whether the LP relaxation of each formulation can underestimate the output of its own trained network.

For the ICNN-MIP formulation, it follows that such underestimation cannot occur when the output is minimised, giving a zero root gap in this setting. In contrast, no analogous guarantee holds for a standard FNN-MIP formulation. Since FNN weights are unrestricted in sign, the induction argument underlying Proposition~\ref{prop:exactness} breaks down. The relaxed values of the intermediate variables $z_k$ can interact with negative weight entries in subsequent layers and produce an underestimation of the surrogate output. The FNN-MIP relaxation is therefore only guaranteed to provide a lower bound on its own exact MIP optimum. 

To illustrate the distinction established above, we approximate $f(x) = x^2$ over $[-2,2]$ using both a standard FNN and an ICNN with comparable accuracy (two hidden layers with four neurons each). 
For each surrogate, we construct the corresponding MIP formulation with objective $\min z_K$. For each $x \in [-2,2]$ on a uniform grid with step size 0.01, we set $z_0 = x$ in each MIP formulation, solve the resulting LP relaxation, and read off the output $z_K$, denoted $f^{\mathrm{LP}}_\mathrm{ICNN}(x)$ and $f^{\mathrm{LP}}_\mathrm{FNN}(x)$ respectively.
The left panel of Figure~\ref{fig:lp_gap} shows that $f^{\mathrm{LP}}_\mathrm{FNN}(x)$ underestimates $\hat f_\mathrm{FNN}(x)$, whereas $f^{\mathrm{LP}}_\mathrm{ICNN}(x)$ coincides with $\hat f_\mathrm{ICNN}(x)$ across the domain, confirming Proposition~\ref{prop:exactness}. 
The right panel quantifies the resulting pointwise relaxation gap. The FNN-MIP exhibits a substantial gap, while the ICNN-MIP gap is zero throughout.

\begin{figure}[ht!]
    \centering
    \includegraphics[width=0.95\linewidth]{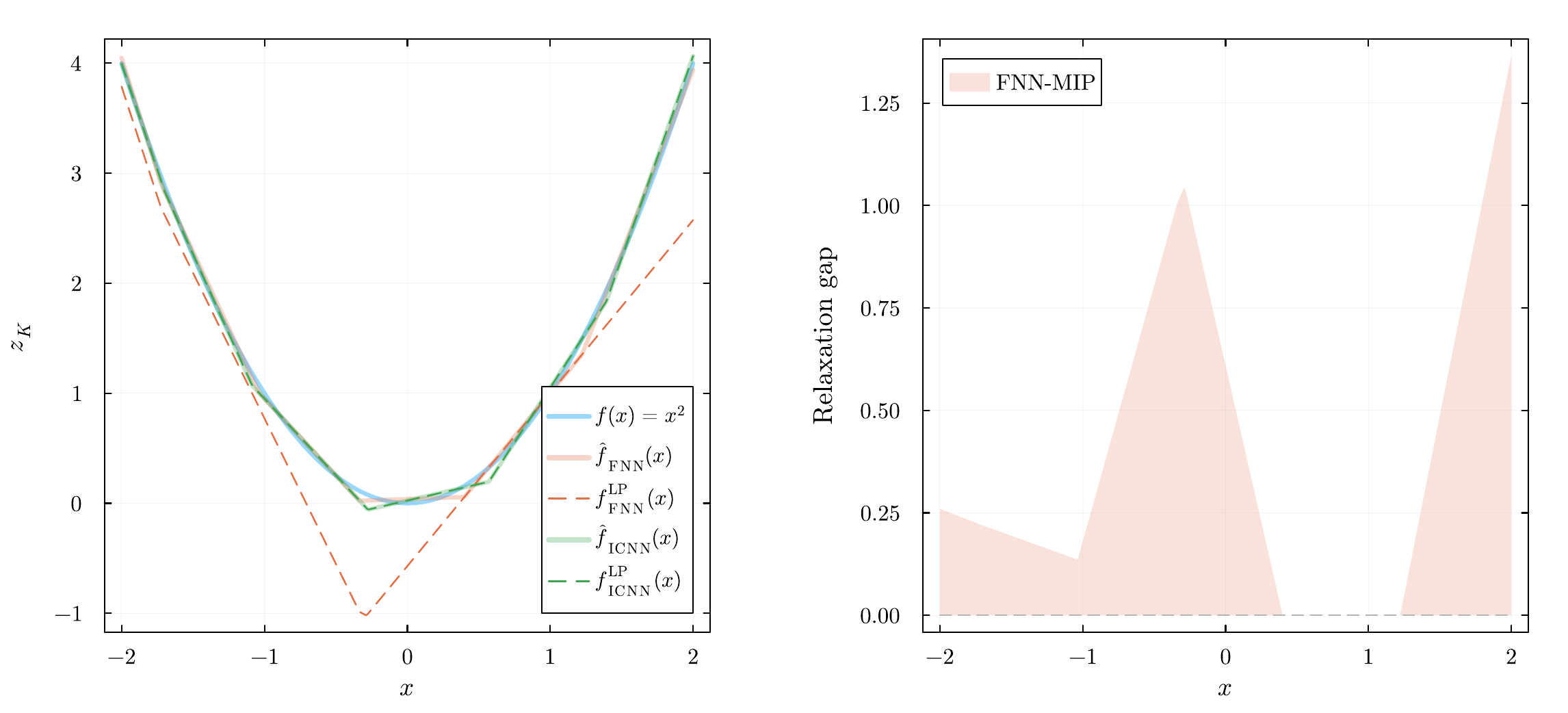}
    \caption{Neural surrogates and their LP relaxations for $f(x)=x^2$ over $[-2,2]$ under direct output minimisation. Left: surrogate outputs and LP relaxation envelopes. Right: pointwise relaxation gap for FNN-MIP, with a dashed line marking zero gap for ICNN-MIP.}
    \label{fig:lp_gap}
\end{figure}

In more general embeddings, where the surrogate output is not minimised with positive costs, or where additional constraints involving the surrogate output prevent the optimiser from driving it to its smallest feasible value, the LP relaxation of the ICNN-MIP embedding may no longer be exact. Nevertheless, because the ICNN architecture itself benefits from a tight LP relaxation, it tends to result in stronger LP relaxations than FNN-MIP formulations. This provides a structural rationale for the tighter root node relaxations, fewer branching decisions, and faster convergence that ICNN-MIP formulations can exhibit, as observed in the computational experiments in Section~\ref{sec:exp}.

\section{A specialised branch-and-bound algorithm}
\label{sec:icnn_bb}

This section develops the second computational route introduced in Section~\ref{sec:pre_embedding}, which is specific to ICNNs. Instead of encoding activations with binary variables, each ReLU unit can be replaced by its LP epigraph representation, yielding an LP-based reformulation of the embedded problem~\eqref{eq:orig}. We first formalise the conditions under which such a reformulation is valid. For the cases where it is not, we construct the strongest continuous relaxation of the ICNN's graph by combining the epigraph with a concave envelope and develop a specialised BB algorithm that exploits this relaxation to achieve convergence via LP solves throughout the search tree.

\subsection{Validity of the epigraph embedding} \label{sec:validity}

Integrating the epigraph representation of the ICNN into the general embedded problem~\eqref{eq:orig} gives%
\begin{mini!}
    {x,y}{h(x,y) }{\label{eq:icnn-full}}{}
    \addConstraint{g(x,y)}{ \leq 0, \quad}{}
    \addConstraint{x}{= z_0, \quad}{}
    \addConstraint{z_{k+1}}{\geq W_k z_k + S_kz_0 + b_k, \quad}{k = 0, \ldots, K-1,}
    \addConstraint{z_k}{\geq 0, \quad}{k = 1, \ldots, K-1,}
    \addConstraint{y}{=z_K, \quad}{}
    \addConstraint{x}{\in \mathcal{X}, \quad}{}
\end{mini!}
where $z_0,\dots,z_K$ are auxiliary variables and $\{W_k,S_k,b_k\}_{k=0}^{K-1}$ are trained NN parameters.

This embedding introduces the ICNN only through linear constraints. However, because the relation $y=\hat f(x)$ in~\eqref{eq:orig} is replaced by epigraph inequalities $y\geq\hat f(x)$, the feasible region of~\eqref{eq:icnn-full} is a relaxation of~\eqref{eq:orig}. Whether this relaxation is tight, and hence whether~\eqref{eq:icnn-full} is a valid reformulation of~\eqref{eq:orig}, is the central question. Definition~\ref{def:validity} formalises what we mean by the epigraph formulation being valid.%

\begin{definition}[validity of the LP epigraph embedding]\label{def:validity}
The epigraph embedding~\eqref{eq:icnn-full} is valid with respect to the original problem~\eqref{eq:orig} if
\begin{enumerate}
    \item[(i)] the two problems have the same optimal objective value, and
    \item[(ii)] the optimal solution $(x^*,y^*)$ of~\eqref{eq:icnn-full} satisfies $y^* = \hat{f}(x^*)$, i.e., the ICNN constraint is active at optimality.
\end{enumerate}
\end{definition}

This validity issue is not specific to the LP epigraph embedding. It also arises in the LP relaxation of the ICNN-MIP formulation~\eqref{eq:icnn_mip}. As shown in Corollary~\ref{cor:epigraph_connection}, when $z_K$ is minimised, the LP relaxation induces the same pointwise value as the ICNN epigraph representation and recovers the exact ICNN output. Thus, under direct minimisation of $\gamma z_K$, both formulations are exact. Any overestimation of $z_K$ is penalised by the objective, and the optimisation is driven to the boundary $z_K=\hat f(z_0)$. In more general settings, the surrogate output may appear in equality constraints, lower-bound constraints, coupling constraints, or objective terms that reward larger values of $y$, allowing the optimiser to exploit the slack in the epigraph inequalities. In other words, it may admit solutions where $z_K > \hat{f}(z_0)$, meaning that the surrogate output is overestimated. 

Although the equivalence between the LP relaxation of the ICNN-MIP and the epigraph embedding is preserved, the LP relaxation of the ICNN-MIP is no longer necessarily exact with respect to the ICNN-MIP formulation itself. That is, the relaxed activation variables may allow for $z_k$ values that do not correspond to the exact ICNN forward pass. Hence, both the LP epigraph embedding and the LP relaxation of the ICNN-MIP formulation may fail to enforce the relation $y=\hat f(x)$, and~\eqref{eq:icnn-full} becomes an invalid reformulation of~\eqref{eq:orig}.

The following result makes the validity condition operational. If the equality $y = \hat f(x)$ in~\eqref{eq:orig} can be equivalently replaced by $y \geq \hat f(x)$, that is, if the replacement preserves both conditions (i) and (ii) of Definition~\ref{def:validity}, then the epigraph embedding~\eqref{eq:icnn-full} is valid. This holds, for instance, when $h$ and $g$ are non-decreasing in $y$, since any slack in $y \geq \hat f(x)$ can then be removed without worsening the objective or violating feasibility. 

In practice, whether the replacement is admissible for a given problem instance is not easy to verify a priori. To detect and, if necessary, close the gap between the epigraph relaxation and the true ICNN output, we construct a concave envelope that provides an upper bound on the ICNN output, complementing the lower bound from the epigraph. Together, these two bounds form the basis of a specialised BB algorithm.

\subsection{Construction of the strongest continuous relaxation}

This section constructs the concave envelope from the convexity of $\hat f$ and combines it with the epigraph embedding. The resulting relaxation is the strongest continuous relaxation of the ICNN formulation over the box domain (Theorem~\ref{thm:strongest_relax}).

\subsubsection{A concave envelope for ICNN outputs}

Over the box domain $\mathcal X$ introduced above, $\hat f$ is bounded above by convex combinations of its values at the vertices of $\mathcal X$. We begin by defining notation for the vertex set.

\begin{definition}[Vertices of the box domain]\label{def:vertices}
Let the set of points $\{\tilde{x}^i\}_{i=1}^{2^n} \subset \mathbb{R}^n$ be defined by all the different combinations of the variables' upper and lower bounds. That is, $\{\tilde{x}^i\}_{i=1}^{2^n}$ are the $2^n$ extreme points (vertices) of $\mathcal{X}$.
\end{definition}

Given the vertex set, the concave envelope of $\hat{f}$ admits an explicit lifted representation as the upper bound defined by convex combinations of vertex values, as formalised below.

\begin{proposition}[Concave envelope of the ICNN function]
\label{prop:concave_env}
Let $\hat{f}:\mathcal{X}\to\mathbb{R}$ denote the trained ICNN over the box $\mathcal{X}$, and let $\{\tilde{x}^i\}_{i=1}^{2^n}$ denote the vertices of $\mathcal{X}$ (Definition~\ref{def:vertices}). 
The concave envelope of $\hat{f}$ over $\mathcal{X}$ is the value function of the following LP:
\begin{maxi!}
    {\alpha,\,\bar{z}}{\bar{z}}{\label{eq:concave_env}}{\bar{z}(x) =}
    \addConstraint{\sum_{i=1}^{2^n}\alpha_i\tilde{x}^i}{= x,}{\label{eq:ce_convex_comb}}
    \addConstraint{\sum_{i=1}^{2^n}\alpha_i}{= 1,}{\label{eq:ce_partition}}
    \addConstraint{\bar{z}}{= \sum_{i=1}^{2^n}\alpha_i\,\hat{f}(\tilde{x}^i),}{\label{eq:ce_output}}
    \addConstraint{\alpha_i}{\geq 0,\quad}{ i=1,\ldots,2^n, \label{eq:ce_nonneg}}
\end{maxi!}
where $(\alpha_1,\ldots,\alpha_{2^n})\in\mathbb{R}^{2^n}$ are auxiliary variables and $\hat{f}(\tilde{x}^i)$ are precomputed ICNN outputs at the vertices. 
\end{proposition}

\begin{proof}
Since every $x \in \mathcal{X}$ can be written as a convex combination of the vertices $\{\tilde{x}^i\}_{i=1}^{2^n}$, the LP is always feasible on $\mathcal{X}$.
For any feasible $\alpha$, convexity of $\hat{f}$ and~\eqref{eq:ce_convex_comb}--\eqref{eq:ce_output} give
\begin{equation}
    \bar{z} = \sum_i \alpha_i \hat{f}(\tilde{x}^i)
            \geq \hat{f}\!\left(\sum_i \alpha_i \tilde{x}^i\right) = \hat{f}(x),
\end{equation}
so $\bar{z}(x) \geq \hat{f}(x)$ for every $x \in \mathcal{X}$.
To see that $\bar{z}(\cdot)$ is concave, take any $x_1, x_2 \in \mathcal{X}$ and $\theta \in [0,1]$, and let $\alpha^1$, $\alpha^2$ be optimal for $x_1$ and $x_2$ respectively. Set $\alpha = \theta\alpha^1 + (1-\theta)\alpha^2 \geq 0$, and it is feasible for $x_\theta = \theta x_1 + (1-\theta)x_2$ since
\begin{align}
    &\textstyle\sum_i \alpha_i \tilde{x}^i
        = \theta\sum_i \alpha^1_i \tilde{x}^i
        + (1-\theta)\sum_i \alpha^2_i \tilde{x}^i
        = \theta x_1 + (1-\theta)x_2 = x_\theta, \\
    &\textstyle\sum_i \alpha_i
        = \theta\sum_i \alpha^1_i
        + (1-\theta)\sum_i \alpha^2_i
        = \theta \times 1 + (1-\theta) \times 1 = 1.
\end{align}
The value $\bar{z}(x_\theta)$ is the maximum over all feasible solutions and is therefore at least as large as the objective at this particular feasible $\alpha$, giving
\begin{equation}
    \bar{z}(x_\theta) \geq \sum_i\alpha_i\hat{f}(\tilde{x}^i)
        = \theta\bar{z}(x_1) + (1-\theta)\bar{z}(x_2),
\end{equation}
which establishes concavity of $\bar{z}(\cdot)$.
To show tightness, let $\varphi$ be any concave function satisfying $\varphi(x) \geq \hat{f}(x)$ for all $x \in \mathcal{X}$. In particular, $\varphi(\tilde{x}^i) \geq \hat{f}(\tilde{x}^i)$ at every vertex. For any $x \in \mathcal{X}$, let $\alpha^*$ be the optimal solution of the LP at $x$. By concavity of $\varphi$,
\begin{equation}
    \varphi(x) = \varphi\!\left(\sum_i \alpha^*_i \tilde{x}^i\right)
        \geq \sum_i \alpha^*_i \varphi(\tilde{x}^i)
        \geq \sum_i \alpha^*_i \hat{f}(\tilde{x}^i) = \bar{z}(x),
\end{equation}
so no concave overestimator is tighter than $\bar{z}(\cdot)$.
\end{proof}

Formulation~\eqref{eq:concave_env} requires $2^n$ auxiliary variables $\alpha_i$ and the evaluation of the ICNN at all $2^n$ vertices. Each vertex evaluation amounts to a forward pass through the network and is computationally cheap. The dominant scaling of the construction is therefore governed by the input dimension $n$ of the ICNN rather than the number of neurons, as is the case in classical MIP formulations. 
We note that this tractable construction of the concave envelope is a distinctive property of the ICNN architecture and does not extend to general ReLU FNNs, whose non-convex output makes the construction of tight concave overestimators an open problem.

\subsubsection{Augmented epigraph embedding}

Augmenting the epigraph embedding~\eqref{eq:icnn-full} with the concave envelope~\eqref{eq:concave_env} yields a problem in which the ICNN output $z_K$ is sandwiched between a lower bound (from the epigraph) and an upper bound (from the concave envelope):%
\begin{mini!}
    {x,y,\alpha}{h(x,y)}{\label{eq:envelope_relax}}{}
    \addConstraint{g(x,y)}{ \leq 0,}{\label{eq:er_orig_constr}}
    \addConstraint{y}{= z_{K},}{\label{eq:er_output_link}}
    \addConstraint{x}{= z_0,}{\label{eq:er_input_link}}
    \addConstraint{z_{k+1}}{\geq W_k z_k + S_k z_0 + b_k, \quad}{k=0, \ldots ,K-1,  \label{eq:er_epigraph_affine}}
    \addConstraint{z_k}{\ge 0, \quad}{k=1, \ldots ,K-1,  \label{eq:er_epigraph_nonneg}}
    \addConstraint{x}{= \sum_{i=1}^{2^n}\alpha_i\tilde{x}^i,}{\label{eq:er_convex_comb}}
    \addConstraint{\sum_{i=1}^{2^n}\alpha_i}{= 1,}{\label{eq:er_partition}}
    \addConstraint{z_{K}}{\leq \sum_{i=1}^{2^n} \alpha_i \hat{f}(\tilde{x}^i),}{\label{eq:er_envelope}}
    \addConstraint{\alpha_i}{\geq 0, \quad}{i=1,\dots,2^n, \label{eq:er_nonneg}}
    \addConstraint{x}{\in \mathcal{X}.}{}
\end{mini!}

The role of constraints~\eqref{eq:er_convex_comb}--\eqref{eq:er_nonneg} is to impose the concave-envelope upper bound described in Proposition~\ref{prop:concave_env}. Although the envelope value $\bar z(x)$ is defined by a maximisation problem in~\eqref{eq:concave_env}, the condition $z_K \leq \bar z(x)$ for a given pair $(x,z_K)$ does not require solving this maximisation explicitly inside~\eqref{eq:envelope_relax}. Instead, it can be represented by requiring the existence of a convex decomposition of the same input $x$ such that $z_K$ is no larger than the corresponding convex combination of vertex values, as enforced by~\eqref{eq:er_convex_comb}--\eqref{eq:er_nonneg}. 
Indeed, for a given pair $(x,z_K)$, if such an $\alpha$ exists, then
\begin{equation}
    z_K \leq \sum_i \alpha_i \hat f(\tilde x^i) \leq \bar z(x),
\end{equation}
where the first inequality follows from~\eqref{eq:er_envelope} and the second from the definition of $\bar z(x)$ as the maximum in~\eqref{eq:concave_env}.
Conversely, suppose that a given pair $(x,z_K)$ satisfies $z_K\leq \bar z(x)$, and let $\alpha^*$ be an optimal solution of~\eqref{eq:concave_env} for this input $x$. Then $\alpha^*$ satisfies~\eqref{eq:er_convex_comb}, \eqref{eq:er_partition}, and
\eqref{eq:er_nonneg}. Moreover,
\begin{equation}
    z_K \leq \bar z(x) = \sum_i \alpha_i^* \hat f(\tilde x^i),
\end{equation}
so~\eqref{eq:er_envelope} also holds. Thus, the same $\alpha^*$ satisfies the envelope constraints in~\eqref{eq:envelope_relax}. 
It follows that, for any given pair $(x,z_K)$ with $x\in\mathcal X$, there exists $\alpha$ satisfying the envelope constraints~\eqref{eq:er_convex_comb}--\eqref{eq:er_nonneg} if and only if $z_K\leq \bar z(x)$. Therefore, the concave-envelope upper bound on $z_K$ is exactly imposed.

Together with the ICNN epigraph constraints~\eqref{eq:er_input_link}--\eqref{eq:er_epigraph_nonneg}, this gives a relaxation in which $z_K$ is bounded below by the epigraph representation and above by the concave envelope. The formulation therefore introduces no binary variables or nonlinearities beyond those already present in~\eqref{eq:orig}, and serves as the node subproblem for the proposed BB algorithm. 
Moreover, this two-sided bounding is not merely one valid bounding scheme. It is the strongest continuous relaxation available over the box domain, as the following result formalises.

\begin{theorem}[Strongest continuous relaxation of the ICNN's graph] \label{thm:strongest_relax}
Let $\hat{f}:\mathcal{X}\to\mathbb{R}$ denote the trained ICNN over the box $\mathcal{X}$, and let $\bar z(\cdot)$ denote its concave envelope over $\mathcal{X}$ as characterised in Proposition~\ref{prop:concave_env}. Define
\begin{equation}
    \mathcal{F} = \{(x,y) \in \mathcal{X} \times \mathbb{R} :
    \exists\, (z, \alpha) \text{ satisfying }
    \eqref{eq:er_output_link}\text{--}\eqref{eq:er_nonneg}\}
    \label{eq:feasible_proj}
\end{equation}
as the projection onto $(x,y)$ of the feasible set defined by the epigraph and envelope constraints of~\eqref{eq:envelope_relax}. Then
\begin{equation}
    \mathcal{F}
    = \{(x,y) : x\in\mathcal{X},\ \hat{f}(x) \leq y \leq \bar{z}(x)\}
    = \operatorname{conv}\left(\{(x, \hat{f}(x)) : x \in
    \mathcal{X}\}\right),
    \label{eq:hull_identity}
\end{equation}
that is, $\mathcal{F}$ is the convex hull of the graph of $\hat{f}$ over $\mathcal{X}$. Consequently, \eqref{eq:envelope_relax} is the strongest continuous relaxation of~\eqref{eq:orig} obtainable by relaxing the ICNN constraint alone.
\end{theorem}

\begin{proof}
Let $\mathcal{S} = \{(x,y) : x\in\mathcal{X},\ \hat{f}(x) \leq y \leq \bar{z}(x)\}$  and let $\mathcal{G} = \{(x, \hat{f}(x)) : x \in \mathcal{X}\}$ denote the graph of $\hat f$ over $\mathcal{X}$, so that~\eqref{eq:hull_identity} reads $\mathcal{F} = \mathcal{S} = \operatorname{conv}(\mathcal{G})$.

For the first equality, any feasible point of~\eqref{eq:feasible_proj} satisfies $\hat{f}(x) \leq y \leq \bar{z}(x)$, so $\mathcal{F} \subseteq \mathcal{S}$. Conversely, any $(x,y) \in \mathcal{S}$ belongs to $\mathcal{F}$, since the epigraph constraints admit any $y \geq \hat{f}(x)$ (Corollary~\ref{cor:epigraph_connection}, $z_K$ being bounded only from below) and the envelope constraints admit any $y \leq \bar{z}(x)$, as shown following~\eqref{eq:envelope_relax}. 
As the two constraint groups share only the variables $x$ and $z_K$, the corresponding feasible points combine into a single feasible point for~\eqref{eq:feasible_proj}. Hence $\mathcal{S} \subseteq \mathcal{F}$ and, thus, $\mathcal{F} = \mathcal{S}$.

For the second equality, the set $\mathcal{S}$ is convex, since $\hat{f}$ is convex and $\bar{z}(\cdot)$ is concave (Proposition~\ref{prop:concave_env}), and $\mathcal{G} \subseteq \mathcal{S}$, so $\operatorname{conv}(\mathcal{G}) \subseteq \mathcal{S}$. Conversely, take any $(x, y) \in \mathcal{S}$ and let $\alpha^*$ be an optimal solution of~\eqref{eq:concave_env} at $x$, so that $(x, \bar{z}(x)) = \sum_{i=1}^{2^n} \alpha^*_i (\tilde{x}^i, \hat{f}(\tilde{x}^i))$ is a convex combination of points of $\mathcal{G}$. 
Since $\hat{f}(x) \leq y \leq \bar{z}(x)$, there exists $\theta \in [0,1]$ such that $y = \theta \hat{f}(x) + (1-\theta)\bar{z}(x)$, and therefore
\begin{equation}
    (x, y) = \theta\, (x, \hat{f}(x)) + (1-\theta) \sum_{i=1}^{2^n} \alpha^*_i\, (\tilde{x}^i, \hat{f}(\tilde{x}^i))
\end{equation}
is a convex combination of points of $\mathcal{G}$, with non-negative weights $\theta$ and $(1-\theta)\alpha^*_i$ summing to one. 
Hence $\mathcal{S} \subseteq \operatorname{conv}(\mathcal{G})$, proving~\eqref{eq:hull_identity}. 

The final claim follows since any convex continuous relaxation of $y = \hat{f}(x)$, $x \in \mathcal{X}$ must contain $\mathcal{G}$, and a convex set containing $\mathcal{G}$ contains $\operatorname{conv}(\mathcal{G}) = \mathcal{F}$. Since \eqref{eq:envelope_relax} attains $\mathcal{F}$ exactly, it is the strongest such relaxation.
\end{proof}

Theorem~\ref{thm:strongest_relax} concerns the convex hull of the input-output graph of the entire ICNN. This is a stronger property than existing tightening techniques for ReLU network formulations provide. The strongest known MIP formulation, the extended formulation of~\citet{anderson_strong_2020}, attains the convex hull of the graph of each individual ReLU neuron, yet the intersection of per-neuron hulls can remain a weak relaxation of the network's input-output mapping. For a general ReLU network, no tractable characterisation of the convex hull of the graph is known, whereas for ICNNs it is delivered by $2^n$ forward passes and linear constraints.

\subsection{Branch-and-bound procedure}

We now describe how the relaxation in~\eqref{eq:envelope_relax} is used within a BB framework. At the root node, the algorithm first solves the epigraph embedding~\eqref{eq:icnn-full} directly, without constructing the concave envelope.  
When a relaxation solution is obtained, its feasibility with respect to the original ICNN constraint is assessed by comparing the output variable $z_K^*$ from the relaxation with the true ICNN output $\hat{f}(x^*)$ obtained by a forward pass at the optimal solution $x^*$. If the relative gap $\epsilon$ is within a prescribed tolerance $\tau$, i.e., 
$$\epsilon= \frac{|z_K^* - \hat{f}(x^*)|}{|\hat{f}(x^*)|} \le \tau,$$
the corresponding solution is accepted. Otherwise, the algorithm branches on the ICNN input variables, forming descendant subproblems based on the envelope relaxation~\eqref{eq:envelope_relax} over the corresponding sub-boxes and progressively tightening the envelope until convergence.
 
Algorithm~\ref{alg:BB} presents the pseudocode. The algorithm maintains a list $\mathcal{L}$ of active subproblems, an incumbent objective value $\overline{h}$, and the corresponding solution $\overline{x}$; these are initialised in Line~\ref{alg:bb_init}. Subproblems are explored in breadth-first order, ensuring that all nodes at a given depth are processed before moving deeper into the search tree. 
At each iteration, a subproblem $S$ is selected from $\mathcal{L}$ and solved~(Lines~\ref{alg:bb_select_s}--\ref{alg:bb_solve}), producing the optimal objective value $h^*$, the optimal solution $x^*$, and the relaxation output $z_K^*$ whenever $S$ is feasible~(Line~\ref{alg:bb_feasi_sol}). When required, a forward pass through the ICNN yields the true surrogate output $\hat{f}(x^*)$, from which the relative gap $\epsilon$ is computed in Line~\ref{alg:bb_gap}. The node is either pruned according to Lines~\ref{alg:bb_prune_infeas}--\ref{alg:bb_prune_opt} or branched upon according to Lines~\ref{alg:bb_select_var}--\ref{alg:bb_children}, as described below.

\begin{algorithm}[ht!]
\caption{Envelope-relaxation-based branch-and-bound for ICNN-embedded optimisation}\label{alg:BB}
\begin{algorithmic}[1]
    \State {\bf Inputs.} Outer problem~\eqref{eq:orig}; trained ICNN $\hat{f}:\mathbb{R}^n\to\mathbb{R}$; input box $\mathcal{X} = \{x \in \mathbb{R}^n : l_j \le x_j \le u_j,\ j=1,\dots,n\}$; tolerance $\tau > 0$.
    \State {\bf Output.} Incumbent solution $(\overline{x}, \overline{h})$.
    \State \textbf{Convention.} The root subproblem $S_0$ over $\mathcal{X}$ uses the epigraph embedding~\eqref{eq:icnn-full}; all descendants use the envelope relaxation~\eqref{eq:envelope_relax} over their current sub-boxes. \label{alg:bb_convention}
    \State {\bf Initialise.} $\mathcal{L} \gets \{S_0\}$; \; $\overline{h} \gets +\infty$; \; $\overline{x} \gets \emptyset$. \label{alg:bb_init}
    \While {$\mathcal{L} \neq \emptyset$} \label{alg:BB_loop}
        \State Select $S$ from $\mathcal{L}$ in breadth-first order, with box $\mathcal{X}_S = \{x : l_j \le x_j \le u_j,\ j=1,\dots,n\}$.  \label{alg:bb_select_s}
        \State Solve $S$. \label{alg:bb_solve}
        \If {$S$ is infeasible} \textbf{continue} \Comment{prune by infeasibility} \label{alg:bb_prune_infeas}
        \EndIf
        \State Obtain optimal solution $(x^*, z_K^*, h^*)$. \label{alg:bb_feasi_sol}
        \If {$h^* \ge \overline{h}$} \textbf{continue} \Comment{prune by bound} \label{alg:bb_prune_bound}
        \EndIf
        \State Evaluate $\hat{f}(x^*)$ by forward pass; compute relative gap $\epsilon$. \label{alg:bb_gap}
        \If {$\epsilon \le \tau$} \Comment{prune by optimality} \label{alg:bb_prune_opt}
            \If {$h^* < \overline{h}$} $\overline{h} \gets h^*$; \; $\overline{x} \gets x^*$
            \EndIf
            \State \textbf{continue}
        \EndIf
        \State \textbf{Branch.} Pick $x_j$ with widest interval $u_j - l_j$; break ties by $x_j^*$ closest to $(l_j + u_j)/2$. \label{alg:bb_select_var}
        \State $m_j \gets (l_j + u_j)/2$; \; $\mathcal{X}' \gets \mathcal{X}_S \cap \{x_j \le m_j\}$, \; $\mathcal{X}'' \gets \mathcal{X}_S \cap \{x_j \ge m_j\}$. \label{alg:bb_bisect}
        \State Evaluate $\hat{f}$ at any vertices of $\mathcal{X}', \mathcal{X}''$ not previously evaluated. \label{alg:bb_vertex_eval}
        \State Form child subproblems $S', S''$ over $\mathcal{X}', \mathcal{X}''$; \; $\mathcal{L} \gets \mathcal{L} \cup \{S', S''\}$. \label{alg:bb_children}
    \EndWhile
    \State \Return $(\overline{x}, \overline{h})$.
\end{algorithmic}
\end{algorithm}

Three pruning rules are applied at each node. 
A node is pruned by \emph{infeasibility}~(Line~\ref{alg:bb_prune_infeas}) if its subproblem has no feasible solution. 
It is pruned by \emph{bound}~(Line~\ref{alg:bb_prune_bound}) if $h^* \ge \overline{h}$, since the lower bound on this box already matches or exceeds the incumbent, so no point inside can improve upon it. 
It is pruned by \emph{optimality}~(Line~\ref{alg:bb_prune_opt}) if $\epsilon \leq \tau$, indicating that the relaxation output is sufficiently close to the true ICNN output at the node's optimum; in this case, the incumbent is updated to $(x^*, h^*)$ whenever the objective improves upon $\overline{h}$.

If no pruning criterion applies, the algorithm branches on the ICNN input variable with the widest current interval $u_j - l_j$, as specified in Line~\ref{alg:bb_select_var}. 
Wider intervals tend to allow larger pointwise overestimation of the ICNN output by the concave envelope. Branching along the widest coordinate, therefore, targets the dimension most likely to contribute to the remaining overestimation at the current relaxation solution. 
When several variables are tied in width, which commonly occurs when multiple surrogates share the same input domain, the tie is broken by selecting the variable whose optimal value $x_j^*$ is closest to the midpoint of its interval. Branching is expected to be most effective at such variables because the overestimation is typically concentrated in the interior of the domain.

The chosen variable is then bisected at the midpoint $m_j = (l_j + u_j)/2$, creating two child subproblems with tightened bounds, as shown in Line~\ref{alg:bb_bisect}. 
Midpoint bisection is preferred as a simple and balanced choice that requires no additional computation to identify a split point while still reducing the overestimation induced by the envelope. 
The vertex set of each child (Definition~\ref{def:vertices}) is obtained by inheriting vertices from the parent box and adding the new vertices introduced on the bisecting face. 
In principle, only the newly introduced vertices need to be evaluated to update the concave-envelope coefficients $\hat{f}(\tilde{x}^i)$, as indicated in Line~\ref{alg:bb_vertex_eval}. 
In our implementation, however, we recompute all vertices of each child box, since the cost of forward passes is modest relative to the overhead of maintaining and querying a vertex cache for the input dimensions considered. 
The resulting child subproblems are then added to the active list $\mathcal L$~(Line~\ref{alg:bb_children}).

When the outer problem embeds several ICNN surrogates $\hat{f}^{(s)}: \mathbb{R}^{n_s} \to \mathbb{R}$, $s=1,\dots,N_s$, with disjoint inputs, the algorithm extends naturally by stacking. The per-surrogate inputs are concatenated into a combined variable $x = (x^{(1)}, \dots, x^{(N_s)})$ living in the product box $\mathcal{X}^{(1)} \times \cdots \times \mathcal{X}^{(N_s)}$, and the envelope relaxation~\eqref{eq:envelope_relax} is augmented with one concave-envelope block per surrogate. Crucially, vertices are enumerated \emph{per surrogate}, so the vertex cost grows as $\sum_s 2^{n_s}$ rather than the intractable $2^{\sum_s n_s}$. Variable selection and bisection operate on the stacked coordinates, applying the widest-interval rule in Line~\ref{alg:bb_select_var} across all $\sum_s n_s$ dimensions, and the node's relative gap is taken as the worst-case across surrogates, $\epsilon = \max_s \epsilon^{(s)}$, so optimality-pruning requires every surrogate output to be sufficiently accurate at the node's optimum.

When the epigraph embedding is valid in the sense of Definition~\ref{def:validity}, the relaxation output $z_K^*$ coincides with $\hat{f}(x^*)$ at the root node, so the algorithm terminates immediately without branching. When the embedding is not valid, the algorithm branches on the ICNN input variables, progressively tightening the concave envelope by halving the box domain. Since the ICNN is continuous and piecewise linear, the concave envelope converges uniformly to $\hat{f}$ as the subproblem domains shrink, ensuring that the relative gap $\epsilon$ eventually falls below any $\tau > 0$. Combined with the pruning rules, this guarantees finite termination.

We also note that the algorithm is highly amenable to parallelisation since subproblems at the same depth of the search tree are independent and can be solved concurrently. In practice, this maps naturally onto multi-threaded or distributed execution.

Compared with BB on the MIP-based formulation~\eqref{eq:icnn_mip}, where branching acts on the intermediate binary variables~$\delta_{k,j}$ encoding the network internally, ICNN-BB branches on the input variables directly. The branching decisions, therefore, range over the $n$ input dimensions rather than over the activation variables, whose number grows with network size.

\section{Computational experiments}
\label{sec:exp}

This section evaluates ICNNs as surrogates in optimisation problems through three real-world case studies drawn from the literature: humanitarian food aid~\citep{maragno_mixed-integer_2023}, oil well routing~\citep{grimstad_relu_2019}, and wine blending~\citep{turner_surrogatelib_2024}. In each instance, we compare three methods: the standard MIP formulation of FNNs (FNN-MIP) based on~\eqref{eq:nn_mip}, the equivalent MIP formulation applied to ICNNs (ICNN-MIP) using~\eqref{eq:icnn_mip}, and our specialised BB algorithm (ICNN-BB).

\subsection{Setup}

Each case study is organised around three questions. The first relates to \emph{approximation quality}: can an ICNN match the predictive accuracy of a standard FNN with the same depth and layer widths, so that the imposed convexity restriction does not introduce a substantial practical cost? The second relates to \emph{MIP formulation tightness}: does ICNN-MIP lead to stronger LP relaxations than the FNN-MIP baseline, and does this translate into faster solution times, as suggested by the analysis in Section~\ref{sec:icnn_mip}? Lastly, we consider \emph{algorithmic contribution}: does the proposed ICNN-BB outperform ICNN-MIP by eliminating binary variables altogether, as developed in Section~\ref{sec:icnn_bb}? To enable a consistent comparison, all three methods are evaluated on identical instances. For each case study, the ICNN/FNN pair shares the same number of layers and neurons per layer, and ICNN-MIP and ICNN-BB use the same trained ICNN.

All experiments were run on a computing platform equipped with an AMD Ryzen 7 6800H processor and 16 GB RAM. The optimisation models were implemented in Julia v1.10.3~\citep{bezanson_julia_2017} using the JuMP v1.26.0 modelling interface~\citep{lubin_jump_2023}. The surrogates were trained using Flux v0.14.25~\citep{innes_flux_2018}. Two solvers were employed: Gurobi v12.0.2~\citep{gurobi_optimization_llc_gurobi_2024} and HiGHS v1.17.0~\citep{huangfu_parallelizing_2018}, with 8 threads, a 3600~s time limit, and otherwise default solver settings. Gurobi was selected as a representative commercial solver, while HiGHS serves as a state-of-the-art open-source alternative. 
For ICNN-BB, we use an optimality tolerance $\tau = 0.01$ across all case studies. The time limit is checked between node solves, so a solve already in progress when the limit is reached is allowed to finish, and the reported wall-clock time may marginally exceed 3600 s. 
Full reproducibility materials are provided at~\url{https://github.com/Lycle/icnn-surrogate-opt}.

\subsection{Case 1: Food aid}

Humanitarian food aid aims to deliver nutritionally complete and palatable food to those in need as efficiently as possible. We adopt the model proposed by~\citet{maragno_mixed-integer_2023}, which builds upon the original formulation by~\citet{peters_nutritious_2021} to support decision-making for the World Food Programme. 

\subsubsection{Problem description}

The food aid optimisation model aims to minimise transportation and procurement costs for a food basket while adhering to certain standards for nutritional value and palatability. The food basket consists of varying amounts of 25 commodities such as wheat flour, milk, salt and sugar. Palatability is inherently difficult to model explicitly and is therefore approximated by an NN surrogate.

The model, presented in~\eqref{eq:food_full}, considers three node types: sources $\mathcal{N}_\mathcal{S}$ where commodities are procured, transshipment locations $\mathcal{N}_\mathcal{T}$, and destinations $\mathcal{N}_\mathcal{D}$ where food aid is distributed. Let $\mathcal{K}$ denote the set of commodities and $\mathcal{L}$ the set of nutrients.
The objective~\eqref{food:obj} minimises the total procurement and transportation costs, with $p_{ik}^{\text{P}}$ denoting the procurement cost (\$/ton) of commodity $k$ at source $i \in \mathcal{N}_\mathcal{S}$, $p_{ijk}^{\text{T}}$ the transportation cost (\$/ton) for shipping commodity $k$ from node $i$ to node $j$, and $F_{ijk}$ the flow of commodity $k$ between nodes $i$ and $j$.
Constraint~\eqref{food:con1} ensures flow balance of each commodity across transhipment nodes, correcting the index sets in equation~(6b) of~\citet{maragno_mixed-integer_2023}. Constraint~\eqref{food:con2} enforces demand satisfaction at destination nodes. Here, $D_i$ is the number of beneficiaries at destination $i$, $x_k$ is the grams of commodity $k$ in the food basket, $T_{\text{days}}$ is the number of feeding days, and $\eta$ converts metric tons to grams.
Constraints~\eqref{food:con3}--\eqref{food:con5} enforce nutritional requirements: $\nu_{kl}$ is the amount of nutrient $l \in \mathcal{L}$ in commodity $k \in \mathcal{K}$, $\bar{\nu}_l$ is the minimum required daily intake of nutrient $l$, and salt and sugar are fixed at recommended dietary levels.
The palatability surrogate $\hat{h}(x)$ appears in~\eqref{food:con6}, with a lower bound $t = 0.5$ imposed in~\eqref{food:con7}.
\begin{mini!}
    {x,y,F}
    {\sum_{i \in \mathcal{N}_\mathcal{S}} \sum_{j \in \mathcal{N}_\mathcal{T} \cup \mathcal{N}_\mathcal{D}} \sum_{k \in \mathcal{K}} 
    p_{ik}^{\text{P}} F_{ijk} 
    + \sum_{i \in \mathcal{N}_\mathcal{S} \cup \mathcal{N}_\mathcal{T}} 
    \sum_{j \in \mathcal{N}_\mathcal{T} \cup \mathcal{N}_\mathcal{D}} 
    \sum_{k \in \mathcal{K}} p_{ijk}^{\text{T}} F_{ijk}
    \label{food:obj}}
    {\label{eq:food_full}}{}
    \addConstraint{\sum_{j \in \mathcal{N}_\mathcal{D} \cup \mathcal{N}_\mathcal{T}} F_{ijk}}
    {= \sum_{j \in \mathcal{N}_\mathcal{S} \cup \mathcal{N}_\mathcal{T}} F_{jik}, \quad}
    {\forall i \in \mathcal{N}_\mathcal{T}, \, \forall k \in \mathcal{K} \label{food:con1}}
    \addConstraint{\sum_{j \in \mathcal{N}_\mathcal{S} \cup \mathcal{N}_\mathcal{T}} \eta F_{jik}}
    {= D_i x_k T_{\text{days}}, \quad}
    {\forall i \in \mathcal{N}_\mathcal{D}, \, \forall k \in \mathcal{K} \label{food:con2}}
    \addConstraint{\sum_{k \in \mathcal{K}} \nu_{kl} x_k}
    {\geq \bar{\nu}_l, \quad}{\forall l \in \mathcal{L} \label{food:con3}}
    \addConstraint{x_{\text{salt}}}{ = 5, \quad}{\label{food:con4}}
    \addConstraint{x_{\text{sugar}}}{ = 20, \quad}{\label{food:con5}}  
    \addConstraint{y }{= \hat{h}(x), \quad}{\label{food:con6}}
    \addConstraint{y }{\geq t, \quad}{\label{food:con7}}
    \addConstraint{x_k, F_{ijk}}{ \geq 0, \quad}
    {\forall i,j \in \mathcal{N}, \, \forall k \in \mathcal{K} \label{food:con8}}
\end{mini!}

\subsubsection{Training performance}

The palatability dataset~\citep{maragno_OptiCL_2021} contains 5000 samples with 25 features (commodity amounts) and one label (palatability), split 80/20 for training and validation. 
Our preliminary assessment using alternative ICNNs indicated an approximately concave response, so all ICNNs are trained on the negated target. For each of six architectures, with two or three hidden layers of 10, 50, or 100 neurons, we train a corresponding ICNN/FNN pair. Training uses the \textit{Adam} optimiser~\citep{kingma_adam_2017} with batch size 32 for 200 epochs, with learning rate 0.001 for FNNs and 0.0015 for ICNNs.
 
Table~\ref{table:food_training} reports training time and validation mean squared error (MSE) for each ICNN/FNN pair. The ICNNs match the FNNs on this task, and for the 3\_10 and 3\_50 networks, the ICNN is in fact more accurate than its FNN counterpart. ICNN training is roughly 20\% slower in wall-clock time because of the weight-projection step, but this cost is incurred once and is negligible relative to downstream solve times. Imposing convexity on the surrogate is therefore essentially free for this problem.

\begin{table}[!ht]
    \centering
    \footnotesize
    \caption{Case 1 (food aid): training time and validation MSE of ICNN/FNN pairs. The network label $L\_N$: $L$ hidden layers of $N$ neurons each.}\label{table:food_training}
    \vspace{4pt}
    \begin{tabular}{llrr}
    \toprule
    Network & Type & Training time (s) & MSE ($\times$10\textsuperscript{-4})\\
    \midrule
    \multirow{2}{*}{2\_10}
        & ICNN     & 3.505& 9.208\\
        & FNN       & 2.901& 7.522\\
    \midrule
    \multirow{2}{*}{3\_10}
        & ICNN     & 4.375& 8.031\\
        & FNN       & 3.462& 12.348\\
    \midrule
    \multirow{2}{*}{2\_50}
        & ICNN     & 6.209& 9.652\\
        & FNN       & 4.159& 7.892\\
    \midrule
    \multirow{2}{*}{3\_50}
        & ICNN     & 7.732& 9.652\\
        & FNN       & 5.463& 10.907\\
    \midrule
    \multirow{2}{*}{2\_100}
        & ICNN     & 13.971& 8.882\\
        & FNN       & 11.885& 6.237\\
    \midrule
    \multirow{2}{*}{3\_100}
        & ICNN     & 24.306& 7.815\\
        & FNN       & 19.799& 6.549\\
    \bottomrule
    \end{tabular}
\end{table}

\subsubsection{Optimisation results}

Tables~\ref{table:food_opt_highs} and~\ref{table:food_opt_gurobi} report the optimisation performance of the three methods using HiGHS and Gurobi, respectively. FNN-MIP scales poorly with surrogate size. Even the moderately sized 2\_50 network already pushes HiGHS past the 3600~s time limit after exploring roughly 760{,}000 nodes, and for every architecture with 50 or more neurons per layer, FNN-MIP reaches the time limit using either solver. Gurobi performs better than HiGHS at the 2\_50 scale, but then fails on all larger networks, exploring up to 8.5 million nodes without closing the gap on the 3\_50 instance. ICNN-MIP scales dramatically better. Using Gurobi, every ICNN-MIP instance solves in under 0.15~s at the root node, regardless of network size. Similarly, using HiGHS, only the 2\_100 network produces nontrivial search effort, and inspection of the solution log confirms that HiGHS's root LP relaxation also matches the MIP optimum. The additional search merely constructs a matching feasible integer solution that Gurobi's root heuristic finds immediately. This is consistent with the theoretical analysis in Section~\ref{sec:icnn_mip}, where the tighter LP relaxation induced by the ICNN architecture enables the MIP solver to largely avoid branching in these instances.

\begin{table}[!ht]
    \centering
    \footnotesize
    \caption{Case 1 (food aid): solve performance of ICNN-BB, ICNN-MIP, and FNN-MIP using HiGHS. Reported are the number of nodes explored by the underlying search, total solving time, average time per node, and the objective value at termination. A dash indicates no feasible solution found within the 3600~s time limit; an asterisk marks a feasible incumbent not proven optimal.}\label{table:food_opt_highs}
    \vspace{4pt}
    \begin{tabular}{llrrrr}
    \toprule
    Network & Method & Nodes & Solving time (s) & Time/node (s) & Objective \\
    \midrule
    \multirow{3}{*}{2\_10}
        & ICNN-BB       & 1   & 0.05& 0.050& 32284.7\\
        & ICNN-MIP     & 1& 0.10& 0.100& 32284.7\\
        & FNN-MIP       & 53& 0.68& 0.013& 32695.2\\
    \midrule
    \multirow{3}{*}{3\_10}
        & ICNN-BB      & 1 & 0.05& 0.050& 32617.8\\
        & ICNN-MIP     & 1& 0.19& 0.190& 32617.8\\
        & FNN-MIP       & 668& 2.19& 0.003& 33086.7\\
    \midrule
    \multirow{3}{*}{2\_50}
        & ICNN-BB       & 1 & 0.06& 0.060& 32043.5\\
        & ICNN-MIP     & 1& 0.56& 0.560&  32043.5\\
        & FNN-MIP       & 759141& 3600.09& 0.005& *32490.9\\
    \midrule
    \multirow{3}{*}{3\_50}
        & ICNN-BB           & 1 & 0.06& 0.060& 32138.2\\
        & ICNN-MIP     & 1& 2.63& 2.630&  32138.2\\
        & FNN-MIP       & 273084& 3600.11& 0.013& -\\
    \midrule
    \multirow{3}{*}{2\_100}
        & ICNN-BB           & 1 & 0.06& 0.060& 32176.6\\
        & ICNN-MIP     & 1375& 8.76& 0.006&  32176.6\\
        & FNN-MIP       & 154116& 3600.07& 0.023& -\\
    \midrule
    \multirow{3}{*}{3\_100}
        & ICNN-BB           & 1 & 0.07& 0.070& 32673.5\\
        & ICNN-MIP     & 1& 2.40& 2.400&  32673.5\\
        & FNN-MIP       & 52300& 3600.04& 0.069& -\\
    \bottomrule
    \end{tabular}
\end{table}

\begin{table}[!ht]
    \centering
    \footnotesize
    \caption{Case 1 (food aid): solve performance of ICNN-BB, ICNN-MIP, and FNN-MIP using Gurobi. Columns and notation as in Table~\ref{table:food_opt_highs}.}\label{table:food_opt_gurobi}
    \vspace{4pt}
    \begin{tabular}{llrrrr}
    \toprule
    Network & Method & Nodes & Solving time (s) & Time/node (s) & Objective \\
    \midrule
    \multirow{3}{*}{2\_10}
        & ICNN-BB       & 1& 0.04& 0.0400& 32284.7\\
        & ICNN-MIP     & 1& 0.02& 0.0200& 32284.7\\
        & FNN-MIP       & 41& 0.28& 0.0067& 32695.2\\
    \midrule
    \multirow{3}{*}{3\_10}
        & ICNN-BB      & 1& 0.04& 0.0400& 32617.8\\
        & ICNN-MIP     & 1& 0.02& 0.0200& 32617.8\\
        & FNN-MIP       & 664& 0.15& 0.0002& 33086.7\\
    \midrule
    \multirow{3}{*}{2\_50}
        & ICNN-BB       & 1& 0.04& 0.0400& 32043.5\\
        & ICNN-MIP     & 1& 0.03& 0.0300&  32043.5\\
        & FNN-MIP       & 246362& 54.30& 0.0002& 32210.4\\
    \midrule
    \multirow{3}{*}{3\_50}
        & ICNN-BB           & 1& 0.04& 0.0400& 32138.2\\
        & ICNN-MIP     & 1& 0.04& 0.0400&  32138.2\\
        & FNN-MIP       & 8567591& 3600.05& 0.0004& *32558.3\\
    \midrule
    \multirow{3}{*}{2\_100}
        & ICNN-BB           & 1& 0.04& 0.0400& 32176.6\\
        & ICNN-MIP     & 1& 0.05& 0.0500&  32176.6\\
        & FNN-MIP       & 4962311& 3600.14& 0.0007& *33180.1\\
    \midrule
    \multirow{3}{*}{3\_100}
        & ICNN-BB           & 1& 0.05& 0.0500& 32673.5\\
        & ICNN-MIP     & 1& 0.12& 0.1200&  32673.5\\
        & FNN-MIP       & 3486818& 3600.08& 0.0010& -\\
    \bottomrule
    \end{tabular}
\end{table}

ICNN-BB terminates at the root node on all architectures using either solver, within 0.04--0.07 s, confirming that the epigraph embedding is exact on this problem. 
Using HiGHS, ICNN-BB is faster than ICNN-MIP even when both terminate at the root node. This difference is consistent with ICNN-BB solving a strictly smaller LP formulation derived directly from the epigraph embedding. Specifically, under the same network size, the epigraph embedding in the ICNN-BB root node only introduces a number of continuous variables equal to the hidden-layer neuron counts. In contrast, ICNN-MIP solves an LP relaxation that retains the relaxed binary indicator and associated slack variables for every neuron in the hidden layers, resulting in a larger root node relaxation problem. Using Gurobi, the same computational distinction applies, but it largely offsets the impact of differences in formulation sizes, making both methods near-instantaneous across all architectures and leaving little room for ICNN-BB to demonstrate an advantage. A more demanding stress test of the algorithm follows in Cases~2 and~3.

Across all methods and NN architectures, the objective values agree to within the variation introduced by differences among the trained surrogates. Moreover, as shown in Table~\ref{table:food_training}, the two widest networks (2\_100, 3\_100) achieve the best validation accuracy among the six architectures, yet these are precisely the ones on which FNN-MIP times out. This highlights an important practical advantage of ICNN-BB. 
When the epigraph embedding is valid, as it is throughout this case study, its computational cost is largely independent of network size, so the user is not forced to compromise on surrogate capacity for tractability.

\subsection{Case 2: Oil well routing}

The second case study, drawn from~\citet{grimstad_relu_2019}, considers oil flow from a set of wells through intermediate manifolds to processing facilities, with the objective of maximising total oil collected at the separators. NN surrogates are used for two classes of nonlinear physical relationships: the well performance curves and the riser pressure drop functions. Relative to Case~1, the problem combines these surrogates with a nontrivial combinatorial structure from binary well-to-manifold routing decisions, making it a more demanding test.

\subsubsection{Problem description}

The problem is defined on a directed acyclic graph (Figure~\ref{fig:oil_graph}). 
Nodes partition into oil wells $N^w$, manifolds $N^m$, and separators $N^s$. Edges are the pipelines $E^d$ from wells to manifolds and the risers $E^r$ from manifolds to separators. The flow on each edge is decomposed into three phases $C=\{\text{oil}, \text{gas}, \text{water}\}$. Each well connects to both manifolds, and each manifold connects to a single separator.

\begin{figure}[ht]
    \centering
    \includegraphics[width=0.70\textwidth]{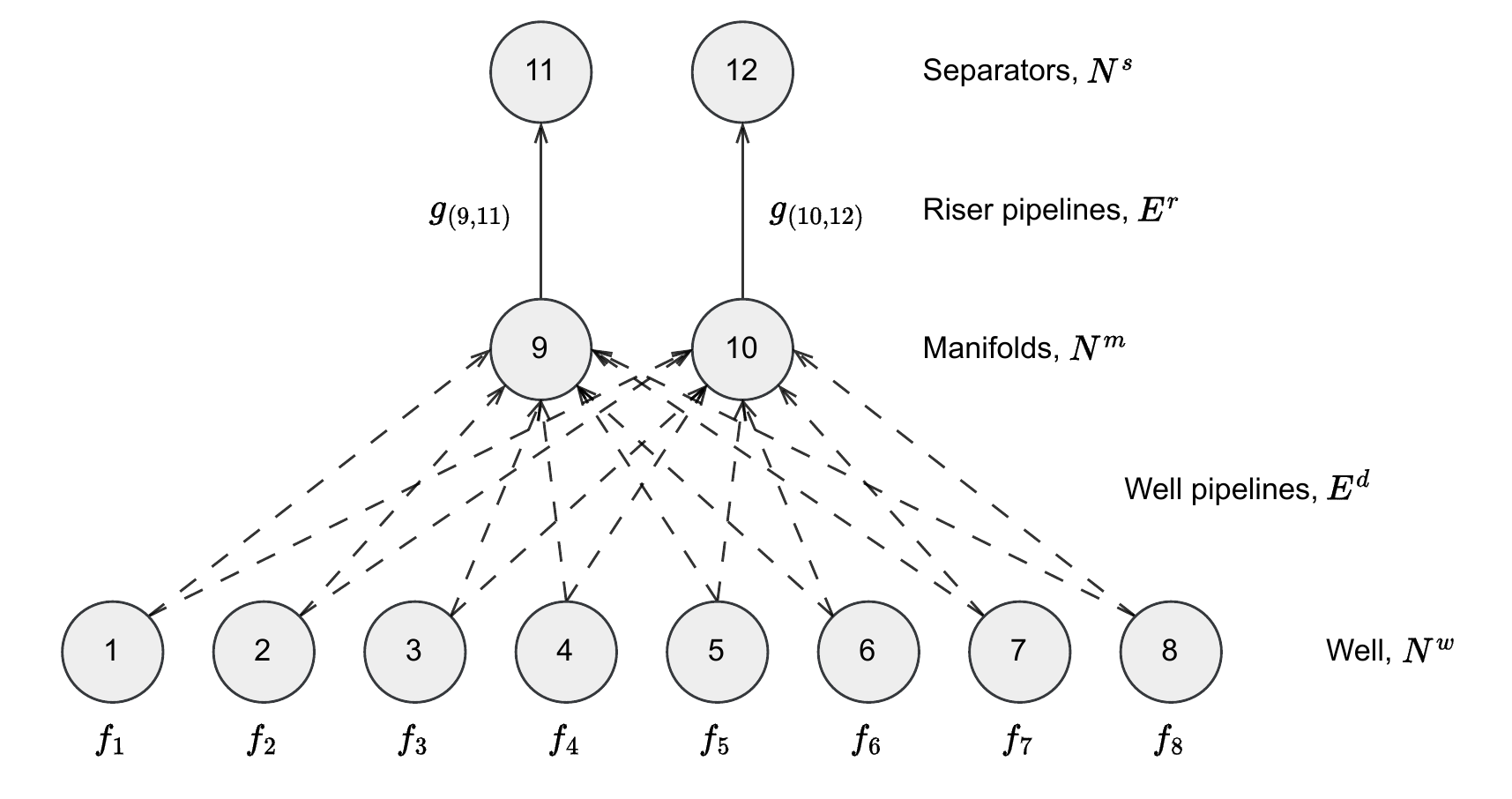}
    \caption{Network structure of the oil production problem with two manifolds and eight wells, adapted from~\citet{grimstad_relu_2019}.}
    \label{fig:oil_graph}
\end{figure}

The oil optimisation problem is formulated as~\eqref{oil_milp}. The objective~\eqref{oil:obj} maximises total oil flow at the separators, where $q_{e, \text{oil}}$ denotes the oil flow on edge $e$. Constraint~\eqref{oil:con1} enforces flow conservation of each phase at the manifolds. Constraint~\eqref{oil:con2} models the riser pressure drop surrogate $\hat g_e$, which predicts the upstream separator pressure $p_j$ from the three-phase flows and the downstream manifold pressure $p_i$, where $e=(i, j)$. The binary variable $y_e$ indicates whether pipe $e \in E^d$ is open. Constraint~\eqref{oil:con3} links the pipe-pressure difference to the routing decision via known pressure limits $p^U$ and $p^L$. Constraint~\eqref{oil:con4} restricts each well to at most one active outgoing pipe. Constraint~\eqref{oil:con5} bounds each phase flow by well-specific limits $q_{e, c}^L$, $q_{e, c}^U$, forcing the flow to zero when $y_e = 0$. Constraint~\eqref{oil:con6} imposes node pressure bounds $p^L$ and $p^U$. 
The well pressure–flow relationships are captured by the surrogates $\hat f_i$ in~\eqref{oil:con7}. Gas and water flows are proportional to oil flow via the known gas-oil and water-oil ratios $c_{e, \text{gor}}$ and $c_{e, \text{wor}}$, enforced by~\eqref{oil:con8}--\eqref{oil:con9}. Separator pressures are fixed at~\eqref{oil:con10}. 
\begin{maxi!}
    {y, q, p}
    {\sum_{e \in E^r} q_{e, \text{oil}} \label{oil:obj}}
    {\label{oil_milp}}
    {}
    \addConstraint{\sum_{e \in E_i^\text{in}} q_{e,c}}{= \sum_{e \in E_i^\text{out}} q_{e,c}, \quad}{\forall c \in C, i \in N^m \label{oil:con1}}
    \addConstraint{p_j}{= \hat g_e(q_{e, \text{oil}}, q_{e, \text{gas}}, q_{e, \text{water}}, p_i), \quad}{\forall e \in E^r \label{oil:con2}}
    \addConstraint{(-p_j^U+p_i^L)(1-y_e)}{\leq p_i-p_j \leq (p_i^U-p_j^L)(1-y_e), \quad}{\forall e \in E^d \label{oil:con3}}
    \addConstraint{\sum_{e \in E_i^\text{out}} y_e}{\leq 1, \quad}{\forall i \in N^w \label{oil:con4}}
    \addConstraint{y_e q_{e, c}^L}{\leq q_{e, c} \leq y_e q_{e, c}^U, \quad}{\forall c \in C, e \in E^d \label{oil:con5}}
    \addConstraint{p_i^L}{\leq p_i \leq p_i^U, \quad}{\forall i \in N \label{oil:con6}}
    \addConstraint{\sum_{e \in E_i^\text{out}} q_{e, \text{oil}}}{= \hat f_i(p_i), \quad}{\forall i \in N^w \label{oil:con7}}
    \addConstraint{\sum_{e \in E_i^\text{out}} q_{e, \text{gas}}}{= c_{e, \text{gor}} \sum_{e \in E_i^\text{out}} q_{e, \text{oil}}, \quad}{\forall i \in N^w \label{oil:con8}}
    \addConstraint{\sum_{e \in E_i^\text{out}} q_{e, \text{water}}}{= c_{e, \text{wor}} \sum_{e \in E_i^\text{out}} q_{e, \text{oil}}, \quad}{\forall i \in N^w \label{oil:con9}}
    \addConstraint{p_i}{= p_i^s, \quad}{\forall i \in N^s \label{oil:con10}}
    \addConstraint{y_e}{\in \{0, 1\}, \quad}{\forall e \in E^d \label{oil:con11}}
\end{maxi!}

The original problem of~\citet{grimstad_relu_2019} has two manifolds and eight wells. To study scaling behaviour, we consider a family of instances $2\_m$ with two manifolds and $m \in \{6, 7, \dots, 12\}$ wells. Each added well contributes one pressure–flow surrogate and two binary routing variables. For $m > 8$ we cycle through the eight base surrogates, so every well remains associated with a pressure–flow curve. Although these enlarged instances do not correspond to configurations from the original study, they preserve the problem structure and provide a controlled test of computational scaling.

\subsubsection{Training performance}

The original problem involves nine surrogates: eight well surrogates and one shared riser surrogate used for both risers, since only a single riser dataset is available in the repository of~\citet{grimstad_relu_2022}. Each well dataset contains 20 to 26 simulator-generated, and therefore noiseless, samples, consisting of a single feature (well pressure) and a single label (oil flow). Given the limited sample size, all points are used for training without a validation split. The riser dataset contains approximately 4000 samples with four features (the three phase flows and manifold pressure) and a single label (separator pressure), split 80/20 for training and validation. Following~\citet{grimstad_relu_2019}, each well surrogate has two hidden layers of 20 neurons, and the riser surrogate has two hidden layers of 50 neurons.

For each of the nine surrogates, we train a corresponding ICNN/FNN pair. Preliminary testing indicated concave relationships in the labelled data for the well response, for which ICNNs were trained on the negated targets, and a convex relationship for riser response, for which ICNNs were trained on the original target. Training used \textit{Adam} optimiser with learning rate 0.001 for both NN classes, with 1000 epochs and batch size 1 for the wells, and 500 epochs and batch size 32 for the riser.

Table~\ref{table:oil_train} reports training time and MSE of each pair, using training MSE for the well surrogates and validation MSE for the riser surrogate. The eight well surrogates show essentially identical accuracy between ICNN and FNN, with MSE on the order of $10^{-5}$; on wells 1, 6, and 8, the ICNN is in fact more accurate. The riser is the one surrogate where the convexity restriction is visibly costly, with the ICNN MSE roughly an order of magnitude larger than that of the FNN. 
Despite this gap, the MSE remains small relative to the output scale, so the ICNN remains acceptable as a physical surrogate for the riser. 

\begin{table}[ht]
    \centering
    \footnotesize
    \caption{Case 2 (oil well routing): training time and reported MSE of ICNN/FNN pairs for the eight well surrogates and the shared riser surrogate.}\label{table:oil_train}
    \vspace{4pt}
    \begin{tabular}{llrr}
    \toprule
    Target & Type & Training time (s) & MSE ($\times$10\textsuperscript{-5})\\
    \midrule
    \multirow{2}{*}{well 1}
        & ICNN     & 2.696& 1.528\\
        & FNN       & 2.336& 2.624\\
    \midrule
    \multirow{2}{*}{well 2}
        & ICNN     & 2.747& 0.240\\
        & FNN       & 2.346& 0.220\\
    \midrule
    \multirow{2}{*}{well 3}
        & ICNN     & 3.506& 2.216\\
        & FNN       & 2.586& 0.600\\
    \midrule
    \multirow{2}{*}{well 4}
        & ICNN     & 3.268& 0.635\\
        & FNN       & 2.813& 0.367\\
    \midrule
    \multirow{2}{*}{well 5}
        & ICNN     & 2.875& 0.110\\
        & FNN       & 2.462& 0.109\\
    \midrule
    \multirow{2}{*}{well 6}
        & ICNN     & 2.785& 0.473\\
        & FNN       & 2.368& 0.708\\
    \midrule
    \multirow{2}{*}{well 7}
        & ICNN     & 3.522& 1.493\\
        & FNN       & 2.940& 0.951\\
    \midrule
    \multirow{2}{*}{well 8}
        & ICNN     & 2.741& 0.690\\
        & FNN       & 2.353& 1.278\\
    \midrule
    \multirow{2}{*}{riser}
        & ICNN     & 3.482& 165.857\\
        & FNN       & 2.992& 18.338\\
    \bottomrule
    \end{tabular}
\end{table}

\subsubsection{Optimisation results}

Tables~\ref{table:oil_opt_highs} and~\ref{table:oil_opt_gurobi} report the performance of the three methods across the $2\_m$ family. As $m$ increases, the number of binary routing decisions, flow variables, and surrogate evaluations also increases, yielding progressively larger mixed-integer models.

\begin{table}[ht]
    \centering
    \footnotesize
    \caption{Case 2 (oil well routing): solve performance of ICNN-BB, ICNN-MIP, and FNN-MIP using HiGHS on instances indexed by $2\_m$, where $m$ denotes the number of downstream routing components. Columns and notation as in Table~\ref{table:food_opt_highs}.}\label{table:oil_opt_highs}
    \vspace{4pt}
    \begin{tabular}{llrrrr}
    \toprule
    Instance & Method & Nodes & Solving time (s) & Time/node (s) & Objective \\
    \midrule
    \multirow{3}{*}{2\_6}& ICNN-BB           & 5315& 57.07& 0.0107& 0.8213\\
        & ICNN-MIP     & 16954& 168.92& 0.0100& 0.8211\\
        & FNN-MIP       & 628591& 3600.06& 0.0057& -\\
    \midrule
    \multirow{3}{*}{2\_7}& ICNN-BB           & 9445& 132.28& 0.0140& 1.0136\\
        & ICNN-MIP     & 39620& 391.62& 0.0099& 1.0136\\
        & FNN-MIP       & 705898& 3600.09& 0.0051& -\\
    \midrule
    \multirow{3}{*}{2\_8}& ICNN-BB           & 19523& 339.23& 0.0174& 1.2168\\
        & ICNN-MIP     & 134940  & 1337.69  & 0.0099& 1.2150 \\
        & FNN-MIP       & 471438 & 3600.04  & 0.0076& - \\
    \midrule
    \multirow{3}{*}{2\_9}& ICNN-BB           & 36689& 758.53& 0.0207& 1.3937\\
        & ICNN-MIP     & 280579& 2716.45& 0.0097& 1.3917\\
        & FNN-MIP       & 330001& 3600.15& 0.0109& -\\
    \midrule
    \multirow{3}{*}{2\_10}& ICNN-BB           & 70859& 1825.59& 0.0258& 1.5627\\
        & ICNN-MIP     & 290818& 3602.68& 0.0124& -\\
        & FNN-MIP       & 307765& 3600.16& 0.0117& -\\
    \midrule
    \multirow{3}{*}{2\_11}& ICNN-BB           & 83911& 3600.88& 0.0429& *1.6084\\
        & ICNN-MIP     & 307805& 3600.16& 0.0117&  *1.5578\\
        & FNN-MIP       & 402570& 3600.15& 0.0089& -\\
    \midrule
    \multirow{3}{*}{2\_12}& ICNN-BB           & 70413& 3701.50& 0.0526& -\\
        & ICNN-MIP     & 278875& 3600.12& 0.0129&  *1.6893\\
        & FNN-MIP       & 377122& 3600.15& 0.0095& -\\
    \bottomrule
    \end{tabular}
\end{table}

\begin{table}[ht]
    \centering
    \footnotesize
    \caption{Case 2 (oil well routing): solve performance of ICNN-BB, ICNN-MIP, and FNN-MIP using Gurobi across the $2\_m$ family. Columns and notation as in Table~\ref{table:food_opt_highs}.}\label{table:oil_opt_gurobi}
    \vspace{4pt}
    \begin{tabular}{llrrrr}
    \toprule
    Instance & Method & Nodes & Solving time (s) & Time/node (s) & Objective \\
    \midrule
        \multirow{3}{*}{2\_6}& ICNN-BB           & 5315& 46.91& 0.0088& 0.8213\\
        & ICNN-MIP     & 5671& 4.97& 0.0009& 0.8212\\
        & FNN-MIP       & 136940& 36.30& 0.0003& 0.8782\\
    \midrule
        \multirow{3}{*}{2\_7}& ICNN-BB           & 9443& 102.31& 0.0108& 1.0136\\
        & ICNN-MIP     & 20395& 8.72& 0.0004& 1.0136\\
        & FNN-MIP       & 326871& 85.03& 0.0003& 1.0762\\
    \midrule
    \multirow{3}{*}{2\_8}& ICNN-BB           & 18143& 241.19& 0.0133 & 1.2168\\
        & ICNN-MIP     & 38779& 33.69& 0.0009& 1.2150 \\
        & FNN-MIP       & 947628& 351.81& 0.0004& 1.2878\\
    \midrule
    \multirow{3}{*}{2\_9}& ICNN-BB           & 38465& 578.17& 0.0150& 1.3937\\
        & ICNN-MIP     & 107448& 71.02& 0.0007& 1.3917\\
        & FNN-MIP       & 2024846& 763.96& 0.0004& 1.4620\\
    \midrule
    \multirow{3}{*}{2\_10}& ICNN-BB           & 70909& 1231.57& 0.0174& 1.5627\\
        & ICNN-MIP     & 172959& 87.56& 0.0005& 1.5624\\
        & FNN-MIP       & 4275671& 1263.28& 0.0003& 1.6147\\
    \midrule
    \multirow{3}{*}{2\_11}& ICNN-BB           & 110809& 2244.76& 0.0203& 1.6693\\
        & ICNN-MIP     & 520406& 200.39& 0.0004&  1.6688\\
        & FNN-MIP       & 5374639& 1660.20& 0.0003& 1.6954\\
    \midrule
    \multirow{3}{*}{2\_12}& ICNN-BB           & 174345& 3723.66& 0.0214& -\\
        & ICNN-MIP     & 642120& 287.76& 0.0004&  1.7297\\
        & FNN-MIP       & 10215414& 3600.02& 0.0004& *1.7299\\
    \bottomrule
    \end{tabular}
\end{table}

The node-count reduction from ICNN-MIP relative to FNN-MIP is consistent across both solvers and all instance sizes. Using HiGHS, FNN-MIP explores hundreds of thousands of nodes on every instance without returning a feasible solution, whereas ICNN-MIP completes every instance up to $2\_9$ within the time limit. Using Gurobi, both methods finish on most instances, but ICNN-MIP's node count is consistently an order of magnitude smaller than FNN-MIP's (up to a factor of sixteen at $2\_{12}$). 
This confirms the theoretical result of Section~\ref{sec:icnn_mip}. The convexity structure of ICNNs tightens the LP relaxation sufficiently to produce a substantially smaller search tree, and the effect is robust across solver choice.

The comparison between ICNN-BB and ICNN-MIP is more nuanced, and the two solvers exhibit different behaviour. 
Using HiGHS, ICNN-BB is the clear winner on all instances where both methods terminate. It explores three to eight times fewer nodes than ICNN-MIP and solves roughly three to four times faster, and on $2\_{10}$, it is the only method to close the gap within the time limit. Using Gurobi, the picture reverses. Gurobi's advanced features drive the cost of an ICNN-MIP node down to roughly $5\times 10^{-4}$ s, more than an order of magnitude cheaper than an ICNN-BB node. ICNN-MIP therefore solves all instances except $2\_{12}$ to optimality using Gurobi, while ICNN-BB is slower in wall-clock time on most instances despite exploring fewer nodes. At the largest $m$, ICNN-BB's per-node cost begins to dominate, and it fails to finish within the time limit using either solver.

The algorithmic advantage of ICNN-BB over ICNN-MIP is therefore solver-dependent. It is pronounced using HiGHS, where generic MIP nodes are expensive, and ICNN-BB's smaller tree pays for itself, but attenuated using Gurobi, where generic MIP nodes are so cheap that the extra per-node cost of the concave-envelope update outweighs the tree-size reduction. In practice, this suggests that ICNN-BB offers the largest benefit in solver environments where generic MIP search is relatively expensive, or on instances whose underlying combinatorial structure remains challenging even for state-of-the-art commercial solvers.

The objective values returned by ICNN-BB and ICNN-MIP agree to within $2\times 10^{-3}$ on every instance where both methods solve to optimality, confirming that ICNN-BB matches ICNN-MIP up to solver tolerance. 
The objective gap between FNN-MIP and the ICNN methods reflects differences between the trained surrogates themselves rather than differences between the optimisation procedures.

\subsection{Case 3: Wine blending}

The third case study, adapted from~\citet{turner_surrogatelib_2024}, concerns a wine producer who purchases grapes from multiple suppliers, each offering grapes with different physicochemical characteristics and prices, and composes several blends that meet a quality threshold at minimum total cost. The quality score of each blend is predicted by an NN surrogate that takes an 11-dimensional physicochemical feature vector as input. This case study, therefore, provides a useful stress test for all three methods under increasing surrogate input dimensionality.

\subsubsection{Problem description}

The wine blending problem is formulated as~\eqref{wine_milp}. Let $n$ denote the number of blends to be created and $m$ the number of grape suppliers. Each supplier $j \in [m]$ provides grapes with feature vector $w_j \in \mathbb{R}^{11}_{+}$, available quantity $\beta_j \ge 0$, and unit cost $c_j \ge 0$. The decision variable $b_{ij} \ge 0$ represents the proportion of grapes from supplier $j$ used in blend $i \in [n]$. The resulting feature vector of blend $i$ is $x_i$, and its predicted quality is $y_i$. 
The objective~\eqref{wine:obj} minimises the total cost of grapes used in all blends. Constraint~\eqref{wine:con1} defines the feature vector of each blend as a linear combination of the features of its constituent grapes. Constraint~\eqref{wine:con2} ensures that the total amount of grapes purchased from each supplier does not exceed the available quantity. Constraint~\eqref{wine:con4} normalises each blend so that the proportions of grapes sum to one. Constraint~\eqref{wine:con5} links each blend to its predicted quality via the trained surrogate $\hat{f}$, and~\eqref{wine:con6} imposes a minimum acceptable quality threshold $q$.
\begin{mini!}
    {x,b,y}{\sum_{i=1}^n\sum_{j=1}^m c_j b_{i,j} \label{wine:obj}}{\label{wine_milp}}{}
    \addConstraint{x_i}{ = \sum_{j=1}^m b_{i,j} w_j, \quad}{\forall i\in[n] \label{wine:con1}}
    \addConstraint{\sum_{i=1}^n b_{i,j}}{\le \beta_j, \quad}{\forall j\in[m] \label{wine:con2}}
    \addConstraint{\sum_{j=1}^m b_{i,j}}{= 1, \quad}{\forall i\in[n] \label{wine:con4}}
    \addConstraint{y_i}{= \hat f(x_i), \quad}{\forall i\in[n] \label{wine:con5}}
    \addConstraint{y_i}{\ge q, \quad}{\forall i\in[n] \label{wine:con6}}
\end{mini!}

To study scaling behaviour, we consider instances $n\_m$ with $n$ blends and $m$ suppliers, where $n \in \{1, 2, 3, 4, 5\}$ and $m = 5n$. Each additional blend introduces one more 11-dimensional surrogate. 

\subsubsection{Training performance}

The wine quality surrogate is trained on the historical dataset of~\citet{cortez_modeling_2009}. 
We use an 80/20 train/validation split. Pre-screening indicated a convex response, so the ICNN is trained on the original target.
A single architecture is used for both the ICNN and the FNN: three hidden layers of 20 neurons each. Training uses the \textit{Adam} optimiser with a learning rate of 0.02 for the FNN and 0.001 for the ICNN, a batch size of 16, and 200 epochs. 

The ICNN reaches validation MSE 0.445, essentially matching the FNN's 0.434, with training wall-clock times of 2.832~s for the ICNN and 2.336~s for the FNN. As in Cases~1 and~2, imposing convexity on the surrogate does not noticeably degrade predictive accuracy on this problem.

\subsubsection{Optimisation results}

Tables~\ref{table:wine_opt_highs} and~\ref{table:wine_opt_gurobi} report performance across the $n\_m$ family. The main scalability axis is $n$, which increases the number of embedded 11-dimensional surrogates, while the proportional growth of $m$ enlarges the underlying blending model through additional supplier-allocation variables and availability constraints.

\begin{table}[!ht]
    \centering
    \footnotesize
    \caption{Case 3 (wine blending): solve performance of ICNN-BB, ICNN-MIP, and FNN-MIP using HiGHS on instances $n\_m$, where $n$ is the number of blends and $m$ the number of suppliers. Columns and notation as in Table~\ref{table:food_opt_highs}.}
    \vspace{4pt}
    \label{table:wine_opt_highs}
    \begin{tabular}{llrrrr}
    \toprule
    Instance & Method & Nodes & Solving time (s) & Time/node (s) & Objective \\
    \midrule
    \multirow{3}{*}{1\_5}
        & ICNN-BB           & 23& 0.39& 0.0171& 1.2552\\
        & ICNN-MIP     & 5& 0.23& 0.0455& 1.2601\\
        & FNN-MIP       & 26& 0.38& 0.0147& 1.2602\\
    \midrule
    \multirow{3}{*}{2\_10}
        & ICNN-BB           & 3& 0.08& 0.0280& 3.0834\\
        & ICNN-MIP     & 1& 1.82& 1.8199& 3.0834\\
        & FNN-MIP       & 279& 2.75& 0.0099& 3.0834\\
    \midrule
    \multirow{3}{*}{3\_15}
        & ICNN-BB           & 3& 0.15& 0.0487& 3.7286\\
        & ICNN-MIP     & 1& 5.11& 5.1142& 3.7286\\
        & FNN-MIP       & 1153& 7.26& 0.0063& 3.7286\\
    \midrule
    \multirow{3}{*}{4\_20}
        & ICNN-BB           & 221& 9.88& 0.0447& 4.8509\\
        & ICNN-MIP     & 1759& 18.28& 0.0104& 4.8509\\
        & FNN-MIP       & 3341& 25.07& 0.0075& 4.8509\\
    \midrule
    \multirow{3}{*}{5\_25}
        & ICNN-BB           & 72791& 3762.44& 0.0517& -\\
        & ICNN-MIP     & 702981& 3600.26& 0.0051&  *6.5086\\
        & FNN-MIP       & 1040126& 3600.14& 0.0035& *6.3888\\
    \bottomrule
    \end{tabular}
\end{table}

\begin{table}[!ht]
    \centering
    \footnotesize
    \caption{Case 3 (wine blending): solve performance of ICNN-BB, ICNN-MIP, and FNN-MIP using Gurobi across the $n\_m$ family. Columns and notation as in Table~\ref{table:food_opt_highs}.}
    \vspace{4pt}
    \label{table:wine_opt_gurobi}
    \begin{tabular}{llrrrr}
    \toprule
    Instance & Method & Nodes & Solving time (s) & Time/node (s) & Objective \\
    \midrule
    \multirow{3}{*}{1\_5}
        & ICNN-BB           & 23& 0.30& 0.0131& 1.2552\\
        & ICNN-MIP     & 1& 0.06& 0.0554& 1.2601\\
        & FNN-MIP       & 1& 0.05& 0.0476& 1.2602\\
    \midrule
    \multirow{3}{*}{2\_10}
        & ICNN-BB           & 1& 0.01& 0.0060& 3.0834\\
        & ICNN-MIP     & 1& 0.03& 0.0274& 3.0834\\
        & FNN-MIP       & 1& 0.03& 0.0341& 3.0834\\
    \midrule
    \multirow{3}{*}{3\_15}
        & ICNN-BB           & 1& 0.01& 0.0080& 3.7286\\
        & ICNN-MIP     & 1& 0.04& 0.0403& 3.7286\\
        & FNN-MIP       & 1& 0.05& 0.0528& 3.7286\\
    \midrule
    \multirow{3}{*}{4\_20}
        & ICNN-BB           & 5& 0.30& 0.0600& 4.8509\\
        & ICNN-MIP     & 1& 0.10& 0.1003& 4.8509\\
        & FNN-MIP       & 1& 0.13& 0.1278& 4.8509\\
    \midrule
    \multirow{3}{*}{5\_25}
        & ICNN-BB           & 72485& 3760.75& 0.0519& -\\
        & ICNN-MIP     & 268442& 204.04& 0.0008& 6.5079\\
        & FNN-MIP       & 50320& 26.29& 0.0005& 6.3888\\
    \bottomrule
    \end{tabular}
\end{table}

The LP-relaxation tightness of ICNN-MIP relative to FNN-MIP, analysed in Section~\ref{sec:icnn_mip}, is preserved in Case 3 using either solver. Using HiGHS, ICNN-MIP explores substantially fewer nodes than FNN-MIP on every instance up to $4\_{20}$: one node versus hundreds or thousands on $2\_{10}$ and $3\_{15}$, and roughly half as many on $4\_{20}$. Using Gurobi, both methods solve instances up to $4\_{20}$ at the root, and the node-count advantage of ICNN-MIP is not visible at that scale. The $5\_25$ instance requires careful interpretation. FNN-MIP solves faster than ICNN-MIP using Gurobi (26~s versus 204~s), but this does not reflect a tighter relaxation. Inspection of the logs shows that the two formulations yield the same root bound ($6.2434$), consistent with the relaxation behaviour discussed in Section~\ref{sec:icnn_mip}. The FNN surrogate admits a lower-cost feasible solution ($6.3888$ versus $6.5079$), due to differences in the response patterns represented by the two fitted surrogates. 

ICNN-BB is the fastest method on the moderate-size instances using HiGHS. On $2\_{10}$ and $3\_{15}$, it completes in under 0.15~s while both MIP methods require seconds, a speedup of roughly one to two orders of magnitude. On $4\_{20}$, the speed advantage narrows, but ICNN-BB remains the fastest method. 
Using Gurobi, every instance up to $4\_{20}$ solves in a fraction of a second. ICNN-BB's per-node cost stays below Gurobi's, but Gurobi proves optimality for these MIPs at its root through its internal procedures. ICNN-BB only overtakes Gurobi when the epigraph embedding is exact at its own root ($2\_{10}$, $3\_{15}$), and otherwise ($1\_{5}$, $4\_{20}$) the accumulated time spent on LP nodes exceeds Gurobi's single root-node solve.

The $5\_{25}$ instance exposes the scalability frontier of ICNN-BB. With five blends each taking an 11-dimensional input, the concave envelope at each node must be constructed over a 55-dimensional box, and the vertex set and LP size grow rapidly with branching depth. Using either solver, ICNN-BB explores around 72{,}000 nodes without returning a feasible incumbent. The MIP formulations fare better at this scale because their per-node cost is not dominated by input dimension, unlike ICNN-BB. 
Therefore, Case~3 identifies a concrete crossover point at which ICNN-BB delivers substantial speedups for small-to-moderate surrogate input dimensions by replacing computationally expensive MIP-based search with relatively inexpensive LP-based subproblem solves, while the exponential growth of the vertex set eventually overwhelms this advantage as the dimension grows.

As in Cases~1 and~2, ICNN-BB and ICNN-MIP objectives agree up to solver tolerance wherever both methods solve to optimality. Any remaining gap between these and FNN-MIP reflects differences among the trained surrogates rather than the optimisation procedures.

\section{Conclusions}
\label{sec:con}

This paper studied the use of input convex neural networks as surrogates in mathematical optimisation. 
We derived the ICNN-MIP formulation and showed that its LP relaxation is exact under direct output minimisation and typically yields stronger relaxations than FNN-MIP in general embedded problems. Combining the epigraph with a concave envelope attains the convex hull of the ICNN's graph over box domains, on which we developed ICNN-BB, a specialised BB algorithm that eliminates the binary variables of the surrogate encoding. It terminates at the root when the epigraph embedding is valid and otherwise branches on the input domain, tightening the relaxation over sub-boxes until convergence. Three case studies from the literature confirm the resulting two-tier advantage, namely, tighter MIP relaxations using standard solvers and efficient LP-driven convergence via ICNN-BB. They also show that the surrogate's input dimensionality sets the practical scalability frontier of ICNN-BB. 

Several directions remain open. The validity condition of the epigraph embedding could be sharpened into an a priori diagnostic for when branching can be avoided on a given trained ICNN. Lower-complexity constructions of the concave envelope, such as sampled, sparsified, or hierarchical alternatives to the $2^n$ vertex enumeration, are a natural avenue for scaling to higher input dimensions, at the cost of relaxations weaker than the convex hull. Further algorithmic refinements, including specialised branching strategies, variable selection, primal heuristics, and parallelisation schemes, could reduce search effort on instances where branching is unavoidable. Replacing the current ad hoc convexity diagnostic with a principled check, such as that of~\citet{gagneux_convexity_2025}, would provide a more rigorous criterion for selecting ICNN surrogates. Finally, the convexity restriction can be relaxed by considering partially convex architectures or ICNN and FNN mixtures, broadening the scope of convexity-exploiting surrogates in mathematical optimisation.

\section*{Acknowledgements}
We thank Eetu Reijonen and Nikita Belyak for their preliminary work on this project. Yu Liu gratefully acknowledges the support from the China Scholarship Council.

\section*{Data availability}
All data and implementations are available in a repository referenced in the paper.

\bibliographystyle{apalike-ejor}
\bibliography{references}

\end{document}